\documentclass[3p]{elsarticle}

\usepackage{bm} 
\usepackage{amsmath} 
\usepackage{amssymb} 
\usepackage{arydshln}
\usepackage{mathdots} 
\usepackage{mathrsfs}
\usepackage{etoolbox} 
\usepackage{color}   
\usepackage{mathtools}
\usepackage{amsthm}
\usepackage{booktabs} 
\usepackage{hyperref}
\graphicspath{{figures/}}

\biboptions{numbers,sort&compress}

\newcommand{\R}{\mathbb{R}}         
\newcommand{\N}{\mathbb{N}}         
\renewcommand{\S}{\mathbb{S}}       

\newcommand{\Ac}{\mathcal{A}}
\newcommand{\Bc}{\mathcal{B}}
\newcommand{\Cc}{\mathcal{C}}
\newcommand{\Dc}{\mathcal{D}}
\newcommand{\Ec}{\mathcal{E}}

\newcommand{\Xc}{\mathcal{X}}

\newcommand{\Ps}{\mathscr{P}}

\newcommand{\Ab}{\mathbf{A}}
\newcommand{\Bb}{\mathbf{B}}
\newcommand{\Cb}{\mathbf{C}}
\newcommand{\Db}{\mathbf{D}}

\newcommand{\Gb}{\mathbf{G}}
\newcommand{\Xb}{\mathbf{X}}

\newcommand{\cge}{\succcurlyeq}
\newcommand{\cle}{\preccurlyeq}
\newcommand{\cl}{\prec}
\newcommand{\cg}{\succ}
\newcommand{\kron}{\otimes}         

\newcommand{\rhi}{\mathrm{RH}_\infty}

\newcommand{\diag}{\mathrm{diag}}
\newcommand{\col}{\mathrm{col}}

\newcommand{\Del}{\Delta}

\newcommand{\ga}{\gamma}
\newcommand{\eps}{\varepsilon}

\newcommand{\1}{\mathbf{1}}

\renewcommand{\t}{\tilde}
\newcommand{\h}{\hat}

\newcommand{\mat}[2]{\left(\begin{array}{@{}#1@{}}#2\end{array}\right)} 
\newcommand{\smat}[1]{\left(\begin{smallmatrix}#1\end{smallmatrix}\right)}

\newcommand{\mstrut}[1]{\rule{0pt}{#1}}
\let\oldhline=\hline 
\renewcommand{\hline}{\oldhline\mstrut{2.5ex}}
\let\oldhdashline=\hdashline 
\renewcommand{\hdashline}{\oldhdashline\mstrut{2.5ex}}

\newcommand{\teq}[1]{\quad\text{#1}\quad} 
\newenvironment{red_test}{\color{red}}{} 

\newtheoremstyle{thmstyle}
{1ex\parskip} 
{1ex\parskip} 
{\normalfont\itshape} 
{} 
{\normalfont\sffamily\bfseries} 
{} 
{.5em} 
{} 

\theoremstyle{thmstyle}
\newtheorem{theorem}{Theorem}[section]

\newtheorem{lemma}[theorem]{Lemma}
\newtheorem{corollary}[theorem]{Corollary}
\newtheorem{definition}[theorem]{Definition}

\newtheorem{remark}[theorem]{Remark}

\newcounter{frstequation}
\newcounter{scndequation}
\newcounter{trdequation}

\makeatletter
\newcommand{\scndlabel}[1]{%
	\protected@write \@auxout {}{\string \newlabel{#1}{{\@currentlabel}{\thepage}{}{equation.\thesection.\the\value{parentequation}\alph{frstequation}}{}} }%
	\hypertarget{equation.\thesection.\the\value{parentequation}\alph{frstequation}}{}
}
\makeatother
\makeatletter
\newcommand{\dlabel}[2]{ 
	\setcounter{frstequation}{\value{equation}}%
	\setcounter{scndequation}{\value{equation}}%
	\addtocounter{scndequation}{1}%
	\edef\theequation{\theparentequation\alph{frstequation},\alph{scndequation}}
	\def\@currentlabel{\theparentequation\alph{frstequation}}\scndlabel{#1}
	\def\@currentlabel{\theparentequation\alph{scndequation}}\label{#2}
	\addtocounter{equation}{1}
}
\makeatother

\makeatletter
\newcommand{\tlabel}[3]{ 
	\setcounter{frstequation}{\value{equation}}
	\setcounter{scndequation}{\value{equation}}%
	\setcounter{trdequation}{\value{equation}}%
	\addtocounter{scndequation}{1}%
	\addtocounter{trdequation}{2}%
	\edef\theequation{\theparentequation\alph{frstequation},\alph{scndequation},\alph{trdequation}}
	\def\@currentlabel{\theparentequation\alph{frstequation}}\scndlabel{#1}
	\def\@currentlabel{\theparentequation\alph{scndequation}}\scndlabel{#2}
	\def\@currentlabel{\theparentequation\alph{trdequation}}\label{#3}
	\addtocounter{equation}{2}
}
\makeatother

\begin{document}

	\begin{frontmatter}
		
		\title{IQC Based Analysis and Estimator Design for Discrete-Time Systems Affected by Impulsive Uncertainties\tnoteref{sponsors}}
		
		\journal{Nonlinear Analysis: Hybrid Systems}
		
		\tnotetext[sponsors]{This project has been funded by Deutsche Forschungsgemeinschaft (DFG, German Research Foundation) under Germany's Excellence Strategy -EXC 2075 -390740016, which is gratefully acknowledged by the authors.}

		\author[address]{Tobias Holicki\corref{mycorrespondingauthor}}
		\cortext[mycorrespondingauthor]{Corresponding author}
		\ead{tobias.holicki@imng.uni-stuttgart.de}
		
		\author[address]{Carsten W. Scherer}

		\address[address]{Department of Mathematics, University of Stuttgart, Pfaffenwaldring 5a, 70569 Stuttgart, Germany}

		\begin{abstract}
			We propose novel quadratic performance tests for linear discrete-time impulsive systems based on viewing these systems as feedback interconnections of some non-impulsive linear system with an impulsive operator.
			In order to systematically analyze such interconnections, we employ the framework of integral quadratic constraints and propose novel constraints of this kind for capturing the behavior of the involved impulsive operator.
			As a major benefit, the modularity of this framework permits seamless extensions to interconnections affected by heterogeneous uncertainties in a straightforward fashion.
			This contrasts with alternative approaches which are based on capturing the system's impulsive behavior by means of a clock.
			%
			Building upon the developed analysis criteria, we characterize the existence of non-impulsive estimators for such impulsive interconnections in a lossless fashion and in terms of linear matrix inequalities.
			Finally, our approach is illustrated by means of several numerical examples.
		\end{abstract}
		
		
	\end{frontmatter}


	\section{Introduction}
	
	Impulsive systems form a rich class of hybrid system which have applications, e.g., in system biology, robotics as well as communication systems, and which have been intensively studied, e.g., in \cite{GoeSan09, HesLib08, YeMic98, BaiSim89, HadChe06, Yan01}.
	They are usually employed for modeling dynamic processes that undergo instantaneous changes at certain events. This class encompasses switched as well as sample-data systems, as shown for example in 
	\cite{Bri17b,NagHes08}.

	Most of the present paper covers the analysis of linear discrete-time impulsive systems where the sequence of impulse instants satisfies some dwell-time condition and does not depend on the state of the underlying system.
	The most common stability tests for such systems are based on the lifting procedure as explained in \cite{CheFra95} in the context of sample-data systems, on the introduction of a clock or timer for capturing the impulsive behavior \cite{Bri13}, or on considering admissible impulse paths similarly as in \cite{XiaTra19} for switched systems.
	Here, we pursue a different route and interpret an impulsive system as a feedback interconnection of a linear non-impulsive system with an impulsive operator. This enables us to employ the framework of integral quadratic constraints (IQCs) \cite{MegRan97} for analyzing such an interconnection.
	This framework is well-known for its modularity and its flexibility in dealing with systems affected by various types of uncertainties or nonlinearities \cite{MegRan97, VeeSch16}. Moreover, the resulting stability or performance tests often, but not always, involve the least conservatism if compared to alternative approaches.
	A related clock-based extension of the IQC framework for analyzing uncertain impulsive systems has been established in \cite{Hol22}. However, to the best of our knowledge, the direct handling of impulsive systems within the IQC framework is unheard of and might pave the way for dealing with several more and otherwise challenging robust design problems.

	As an illustration, we consider in this paper's second part the design of non-impulsive estimators for impulsive systems.
	In the context of impulsive systems, most estimation approaches aim at providing an approximation of the underlying system's state by constructing an impulsive observer which has access to the sequence of impulse instants \cite{MedLaw09, BerSan18, ConPer17}.
	In contrast, we aim to approximate some output of a given impulsive system, which does not necessarily equal its full state, by means of an estimator which does not have access to the sequence of impulse instants. This constitutes a genuine robust estimation problem.
	It is by now well-known that such estimation problems admit convex solutions in various situations \cite{SchKoe08, ScoFro06, Vee15, GerDeo01, Ger99, VenSei16}, and we show that this is also true for our new analysis criteria covering impulsive systems.
	%
	Since these are based on IQCs, our design approach is conceptually related to the ones in \cite{SchKoe08, ScoFro06, Vee15}, but we employ a rather different strategy of proof.
	Technically, we apply the elimination lemma in \cite[Theorem 2]{Hel99} instead of the more commonly used one in \cite[Theorem 3]{Hel99}, which removes the need for any factorization of the dynamic multiplier describing the IQC and, thus, drastically simplifies the required arguments.
	%
	%
	%
	It is somewhat surprising that this has not been noticed so far and its implications
	for other challenging robust or gain-scheduling synthesis problems based on IQCs remains to be explored.
	
	\vspace{1ex}
	
	\noindent\textbf{Outline.} %
	%
	The remainder of the paper is organized as follows. After a short paragraph on notation, we describe the analysis problem for the considered class of feedback interconnections involving an impulsive operator in Section \ref{DI::sec::prob_set_ana}. We also recall various quadratic performance tests available in the literature in Section \ref{DI::sec::recap}.
	In Sections \ref{DI::sec::Recap_IQCs} and \ref{DI::sec::main_ana}, we provide our main analysis results, an novel extension of an IQC Theorem proposed in \cite{SchVee18, Sch21} and the construction of IQCs with nontrivial terminal cost for the underlying impulsive operator.
	In Section \ref{DI::sec::exa_ana}, we compare our new tests with those recapitulated in Section~\ref{DI::sec::recap}.

	Section \ref{DI::sec::es} is organized similarly. The underlying robust estimation problem is described in Section \ref{DI::sec::prob_set_es} and we provide our lossless convex design criteria in Section \ref{DI::sec::main_es}. We recall an alternative synthesis result from \cite{Hol22} based on alternative closed-loop analysis criteria in Section \ref{DI::sec::es_alternative} and compare both approaches in Section \ref{DI::sec::exa_es} by means of a numerical example.

	\vspace{1ex}
	
	\noindent\textbf{Notation.} %
	%
	$\N$ ($\N_0$) denotes the set of positive (nonnegative) integers and $\S^n$ the set of symmetric real $n \times n$ matrices. $I_n$ stands for the $n\times n$ identity matrix and the subscript is omitted if it is not relevant.
	We let $\ell_2^n := \big\{ x\in \ell_{2e}^n\colon \|x\|^2_{\ell_2} := \sum_{t = 0}^\infty x(t)^\top x(t) < \infty \big\}$, where $\ell_{2e}^n$ denotes the set of sequences 
	in $\R^n$.
	Finally, we use the abbreviation
	\begin{equation*}
		\diag(X_1, \dots, X_N) := \mat{ccc}{X_1 &  & 0 \\ & \ddots & \\ 0 & & X_N}
	\end{equation*}
	for matrices $X_1, \dots, X_N$, utilize the Kronecker product $\kron$ as defined in \cite{HorJoh91} and indicate objects that can be inferred by symmetry or are not relevant with the symbol ``$\bullet$''.

	\section{Analysis}
	
	\subsection{Problem Setting}\label{DI::sec::prob_set_ana}

	For real matrices of appropriate dimensions and some initial condition $x(0)\in \R^n$, we consider the discrete-time feedback interconnection
	\begin{subequations}
		\label{DI::eq::sys}	
		\begin{equation}
			\mat{c}{x(t+1) \\ z(t) \\ e(t)} = \mat{ccc}{A & B_w & B \\ C_z & D_{zw} & D_{zd} \\ C & D_{ew} & D}\mat{c}{x(t) \\ w(t)\\ d(t)},\quad
			w(t) = \Del(z)(t)
			\label{DI::eq::sys_intercon}
		\end{equation}
		for $t\in \N_0$, where $d$ denotes some generalized disturbance and $e$ is some output signal that often plays the role of an error term that is desired to be kept small. Here, the operator $\Del$ is characterized by a sequence of impulse instants $-1 = t_0 < t_1 < t_2 < \dots$ and defined by
		\begin{equation}
			\Del(z)(t) := \begin{cases}
				z(t) & \text{ if }t =t_k \text{ for some } k \in \N \\
				0 & \text{ otherwise.}
			\end{cases}
			\label{DI::eq::sys::del}
		\end{equation}
	\end{subequations}
	We emphasize that this description is very flexible and covers, in particular, standard discrete-time linear impulsive systems with a flow and jump (or impulsive) component as described by
	\begin{equation}
		\mat{c}{x(t+1) \\ e(t)} = \mat{cc}{A & B\\ C & D}\mat{c}{x(t) \\ d(t)},\quad
		\mat{c}{x(t_k+1) \\ e(t_k)} = \mat{cc}{A_J & B_{J} \\ C_{J} & D_{J}}\mat{c}{x(t_k) \\ d(t_k)}
		\label{DI::eq::sys_standard}
	\end{equation}
	for $t \in \N_0 \setminus \{t_1, t_2\dots\}$ and $k\in \N$; see, e.g., \cite{LiuLiu07, LiuChe15}.
	Indeed, in \eqref{DI::eq::sys}, one just choose the describing matrices
	\begin{equation*}
		\mat{ccc}{\bullet & B_w & \bullet \\ C_z & D_{zw} & D_{zd}\\ \bullet & D_{ew} & \bullet}
		:= \mat{c|cc|c}{\bullet & A_J - A & B_{J} - B & \bullet\\ \hline I & 0 & 0 & 0\\ 0 & 0 & 0 & I \\ \hline \bullet & C_{J} -C & D_{J} - D & \bullet}.
		%
	\end{equation*}
	%
	%
	Conversely, if $I - D_{zw}$ is nonsingular, then \eqref{DI::eq::sys} translates into  \eqref{DI::eq::sys_standard} with 
	\begin{equation*}
		\mat{cc}{A_J & B_J \\ C_J & D_J} := \mat{cc}{A & B \\ C & D} + \mat{c}{B_w \\ D_{ew}} (I - D_{zw})^{-1}\mat{cc}{C_z & D_{zd}}.
	\end{equation*}
	Let us stress already at this point one of the benefits of the description \eqref{DI::eq::sys} over
	the standard one in \eqref{DI::eq::sys_standard}. In \eqref{DI::eq::sys}, we can easily model problems in which only a subset of the state and output entries are subject to jumps.
This emerges naturally if one thinks of an $H_\infty$ design setup for a complex interconnection where, typically, several dynamic weights are introduced in order to shape the frequency response of the closed-loop. Clearly, then not all states and outputs undergo jumps, which could be beneficial for synthesis. More such aspects will be discussed later.


\vspace{1ex}

In the sequel, we are mostly interested in stability and quadratic performance of the interconnection \eqref{DI::eq::sys} which are defined in a standard fashion as follows.

\begin{definition}[Well-Posedness, Stability and Quadratic Performance]\mbox{}
	\label{DI::def::stab}
	\begin{itemize}
		\item The interconnection \eqref{DI::eq::sys} is said to be well-posed if $\det(I - D_{zw}) \neq 0$ holds.
		\item The interconnection \eqref{DI::eq::sys} is said to be (exponentially) stable if it is well-posed and if there exists constants $M > 0$, $\rho \in (0, 1)$ such that $\|x(t)\| \leq M \rho^t \|x(0)\|$ holds for all $t \in \N_0$ and any initial condition $x(0) \in \R^n$ in case of $d = 0$.
		\item Suppose that $P = \smat{Q & S \\ S^\top & R}\in \S^{n_e +n_d}$ with $Q \cge 0$. Then the interconnection \eqref{DI::eq::sys} is said to achieve quadratic performance with index $P$ if it is stable and if there exists some $\eps > 0$ such that
		\begin{equation*}
			\sum_{t = 0}^\infty (\bullet)^\top P \mat{c}{e(t) \\ d(t)} \leq -\eps \|d\|_{\ell_2}^2
		\end{equation*}
		holds for the initial condition $x(0) = 0$ and for all $d \in \ell_2^{n_d}$.
	\end{itemize}
\end{definition}

It is well-known that stability and performance of the interconnection \eqref{DI::eq::sys} or, with analogous definitions, of the impulsive system \eqref{DI::eq::sys_standard} not only depends on their describing matrices but also heavily on the properties of the sequence of impulse instants $(t_k)_{k\in \N_0}$.
Crucial are the so-called dwell-times $t_{k+1} - t_k -1$, which equal the duration of how long the flow component is active until the next jump occurs. Thus analysis results are usually formulated under bounds on these dwell-times. We will mostly consider the following ones.

\begin{definition}[Dwell-Time Conditions]
	The strictly increasing sequence $(t_k)_{k\in \N_0}$ of integers with $t_0 = -1$ is said to satisfy
	\begin{enumerate}
		\item an arbitrary dwell-time (ADT) condition if this sequence is not further constrained, i.e.,
		\begin{equation*}
			t_{k+1} - t_k - 1 \in \N_0 \teq{ for all }k \in \N_0;
			\label{DI::eq::ADT}
			\tag{ADT}
		\end{equation*}
		\item an exact dwell-time condition (EDT) if there exists some $T \in \N$ such that
		\begin{equation}
			t_{k+1} - t_{k}-1 = T \teq{ for all }k \in \N_0;
			\label{DI::eq::EDT}
			\tag{EDT}
		\end{equation}
		%
		\item a minimum dwell-time condition (MDT) if there exists some $T_{\min} \in \N$ such that
		\begin{equation}
			t_{k+1} - t_{k}-1 \geq T_{\min} \teq{ for all }k\in \N_0;
			\label{DI::eq::MDT}
			\tag{MDT}
		\end{equation}
		\item a range dwell-time condition (RDT) if there exists $0 < T_{\min} \leq T_{\max}$ such that
		\begin{equation}
			t_{k+1} - t_{k}-1 \in [T_{\min}, T_{\max}]
			\teq{ for all }k \in \N_0.
			\label{DI::eq::RDT}
			\tag{RDT}
		\end{equation}
	\end{enumerate}
\end{definition}

Considering sequences of impulse instants with arbitrary dwell-time is often not appropriate because some knowledge about the dwell-times is typically available in practice. Moreover, the resulting analysis criteria might put too severe constraints on the underlying system. For example, they require both matrices $A$ and $A_J$ in \eqref{DI::eq::sys_standard} to be Schur stable.
However, these analysis criteria are often taken as a starting point since they are simple and the least demanding ones in terms computation times.
%

Typically it is beneficial to take available knowledge on the dwell-time into account.
The exact dwell-time condition means that once an impulse occurs, the flow component is active $T$ times until the next impulse occurs. In particular, for a sequence satisfying \eqref{DI::eq::EDT}, the interconnection \eqref{DI::eq::sys} and the system \eqref{DI::eq::sys_standard} are both periodic with period $T+1$ which enables a number of techniques for their analysis.
Both minimum and range dwell-time conditions are practically relevant extensions and lead to aperiodic systems in general.

\subsection{Recap of Some Available Results on Standard Discrete-Time Impulsive Systems}\label{DI::sec::recap}

For reasons of comparison and completeness, we recall in this subsection some of the available stability analysis results for impulsive systems described by \eqref{DI::eq::sys_standard}; most of the corresponding proofs are given in the appendix. For brevity, we only consider results for sequences of impulse instants with \eqref{DI::eq::RDT}, but those for sequences with \eqref{DI::eq::EDT} are easily recovered by taking $T_{\min} := T_{\max} := T$ and, roughly speaking, those for sequences with \eqref{DI::eq::MDT} are obtained by taking $T_{\max} := T_{\min}$ and by including a suitable LMI that ensures quadratic performance of the system's flow component separately.

\vspace{1ex}

The first recapped analysis result 
essentially relies on lifting as presented in \cite{CheFra95} in the context of sampled-data control. In a nutshell, the idea is to express the output $e(t)$ only in terms of the state $x(t_k+1)$ at time $t_k+1$ and the inputs $d(t_k+1), \dots, d(t)$ for $t \in [t_k+1, t_{k+1}]$, which is possible due to the discrete-time variation of constants formula.

\begin{theorem}[Lifting Based Quadratic Performance Test]
	\label{DI::theo::basic_Lyapunov}
	The impulsive system \eqref{DI::eq::sys_standard} achieves quadratic performance with index $P$ for all sequences $(t_k)_{k\in\N_0}$ satisfying \eqref{DI::eq::RDT} if there exists $X \in \S^n$ satisfying
	\begin{subequations}
		\label{DI::theo::eq::basic_Lyapunov_LMIs}
		\begin{equation}
			\dlabel{DI::theo::eq::basic_Lyapunov_LMIsa}{DI::theo::eq::basic_Lyapunov_LMIsb}
			X \cg 0
			\teq{ and }	(\bullet)^\top \mat{cc|c}{X & 0 & \\ 0 & -X & \\ \hline  && I_{k+1} \kron P}
			\mat{cccccc}{A_JA^k & A_JA^{k-1}B & A_JA^{k-2}B & \dots & A_JB & B_J
				\\ I & 0 & \dots & \dots & 0&0\\ \hline C & D& 0 & \dots & 0 & 0\\
				0 & I & 0 & \dots & 0 & 0\\
				C A & C B & D & & \vdots & \vdots\\
				0 & 0 & I &  &  \vdots & \vdots\\
				\vdots & \vdots & \vdots & \ddots &  \vdots & \vdots\\
				CA^{k-1} & CA^{k-2}B & CA^{k-3}B & \dots & D & 0 \\
				0 & 0 & 0 & \dots & I & 0 \\
				C_J A^{k} & C_J A^{k-1} B & C_J A^{k-2} B & \dots & C_JB & D_J \\ 0 & 0 & 0 & \dots & 0 & I} \cl 0
		\end{equation}
	\end{subequations}
	for all $k \in [T_{\min}, T_{\max}] \cap \N$.
\end{theorem}

While the LMIs \eqref{DI::theo::eq::basic_Lyapunov_LMIs} are easily derived and conveniently passed to any SDP solver, the resulting performance test can potentially be conservative.
In the context of switched systems, this has been resolved in \cite{XiaTra19} by Lyapunov arguments involving so-called admissible and postadmissible switching paths of some length $L \in \N$. This idea can be easily modified to apply for impulsive systems \eqref{DI::eq::sys_standard} as well. It has been shown that the resulting test for stability is nonconservative if      $L$ is not fixed a priori.
As the price to be paid, even for a fixed length $L$, this test involves a (much) larger computational burden than the one in Theorem \ref{DI::theo::basic_Lyapunov}.

\begin{theorem}[Path Based Quadratic Performance Test]
	\label{DI::theo::path_Lyapunov}
	Let $L \in \N$ and let $e_1, \dots, e_L$ denote the standard unit vectors in $\R^L$. Define by
	\begin{equation*}
		\Ps_L \!:=\!\left\{\sum_{k = 1}^{m} e_{j_k} ~\colon~ m\in \N,~ j_1 - 1 \leq T_{\max},~ L - j_m \leq T_{\max}\text{ and } j_{k+1} - j_k - 1 \in [T_{\min}, T_{\max}] \text{ for all }k\right\} \subset \R^L
	\end{equation*}
	the set of admissible impulse paths of length $L$ and, for each $p \in \Ps_L$, by
	\begin{equation*}
		\Ps^+_L(p) = \left\{ q \in \Ps_L ~\colon \mat{c}{p \\ q} \in \Ps_{2L}\right\} \subset \Ps_L
	\end{equation*}
	the set of postadmissible impulse paths of length $L$. Then the impulsive system \eqref{DI::eq::sys_standard} achieves quadratic performance with index $P$ for all $(t_k)_{k\in \N_0}$ satisfying \eqref{DI::eq::RDT} if there exist $X_{p_1}, \dots, X_{p_{|\Ps_L|}} \in \S^n$ satisfying
	\begin{subequations}
		\label{DI::theo::eq::path:based:LMIs}
		\begin{equation}
			X \cg 0
			\label{DI::theo::eq::path:based:LMIsa}
		\end{equation}
		and
		\begin{equation}
			\label{DI::theo::eq::path:based:LMIsb}
			(\bullet)^\top\! \mat{cc|c}{X_q & 0 & \\ 0 & \!-X_p & \\ \hline  && I_{L}\! \kron\! P}\vspace{-1ex}
			\mat{ccccc}{A_{p} & A_{(p_2, \dots, p_L)}B_{p_1} & A_{(p_3,\dots, p_L)}B_{p_2} & \dots & B_{p_L}
				\\ I & 0 & \dots & \dots & 0\\ \hline C_{p_1} & D_{p_1}& 0 & \dots & 0 \\
				0 & I & 0 & \dots & 0 \\
				C_{p_2} A_{p_1} & C_{p_2} B_{p_1} & D_{p_2} & & \vdots \\
				0 & 0 & I &  &  \vdots \\
				\vdots & \vdots & \vdots & \ddots &  \vdots\\
				C_{p_L} A_{(p_1,\dots, p_{L-1})} & C_{p_L} A_{(p_2,\dots, p_{L-1})} B_{p_1} & C_{p_L} A_{(p_3,\dots, p_{L-1})} B_{p_2} & \dots & D_{p_L} \\ 0 & 0 & 0 & \dots & I} \cl 0
		\end{equation}
	\end{subequations}
	for all $q \in \Ps_L^+(p)$ and all $p = (p_1, \dots, p_L) \in \Ps_L$. Here, we employ the abbreviations
	\begin{equation*}
		A_v :=  A^{1 - v_r} A_J^{v_r}\cdot \ldots \cdot A^{1 - v_1} A_J^{v_1},
		\quad B_s := B^{1-s}B_J^s, \quad C_s := C^{1-s}C_J^s \teq{ and }D_s := D^{1-s}D_J^s
	\end{equation*}
	for any vector $v = (v_1, \dots, v_r)\in \{0, 1\}^r$ with $r \in \N$ and any scalar $s \in \{0, 1\}$.
\end{theorem}

In a nutshell, the first idea is to exploit that the asymptotic behavior of the impulsive system \eqref{DI::eq::sys_standard} with a vanishing disturbance $d = 0$ and with a sequence $(t_k)_{k\in \N_0}$ of impulse instants satisfying \eqref{DI::eq::RDT} coincides with the one of the system
\begin{equation}
	\label{DI::eq::sysL}
	x(L(k+1)) = A_{p^{(k)}} x(Lk)
\end{equation}
for $k \in \N_0$ and some sequence $(p^{(k)})_{k\in \N_0}$ with elements in $\{0, 1\}^L$ corresponding to $(t_k)_{k\in \N_0}$; this is essentially only a consequence of the employed abbreviations.
Second, one notes that we actually have $p^{(k)} \in \Ps_L$ for all $k$ since we know that the sequence of impulse instants satisfies \eqref{DI::eq::RDT}, i.e., any $p^{(k)}$ must be an admissible impulse path. This usually excludes many of the elements in $\{0, 1\}^L$.
Third, one takes another step and also considers the concatenation of admissible paths. More precisely,  one notes that for a given admissible path only a subset of the paths in $\Ps_L$ results in a concatenated path that is admissible, i.e., contained in $\Ps_{2L}$. In particular, we must have $p^{(k+1)} \in \Ps_{L}^+(p^{(k)})$ for all $k \in \N_0$. 

As an example let us suppose that $T_{\min} = 2$, $T_{\max} = 3$ and $L = 5$. Further, we assume that the first four elements of the sequence $(t_k)_{k\in \N_0}$ are given by $t_0 = -1$, $t_1 = 3$, $t_2 = 6$ and $t_3 = 9$. Then the first three states of the system \eqref{DI::eq::sysL} are given by $x(0)$,
\begin{equation*}
	\arraycolsep=2pt
	x(5) \!=\! AA_JAAAx(0)
	=A^{1-0}A_J^0A^{1-1}A_J^1A^{1-0}A_J^{0}A^{1-0}A_J^{0}A^{1-0}A_J^0x(0)
	= A_{p^{(0)}}x(0)
	\text{ ~for~ }
	p^{(0)} := \mat{ccccc}{0 & 0 & 0 & 1 & 0}^\top
\end{equation*}
and $x(10) = A_JAAA_JAx(5) = A_{p^{(1)}}x(5)$ for $p^{(1)} := (0, 1, 0, 0, 1)^\top$. Moreover, one can check that the set of admissible paths corresponding to $(T_{\min}, T_{\max}, L) = (2, 3, 5)$ is given by
\begin{equation*}
	\Ps_{L} = \left\{\mat{c}{0 \\ 1 \\ 0 \\ 0 \\ 0}, \mat{c}{0 \\ 0 \\ 1 \\ 0 \\ 0}, \mat{c}{0 \\ 0 \\ 0 \\ 1 \\ 0}, \mat{c}{1 \\ 0 \\ 0 \\ 0 \\ 1}, \mat{c}{1 \\ 0 \\ 0 \\ 1 \\ 0}, \mat{c}{0 \\ 1 \\ 0 \\ 0 \\ 1} \right\}
	\teq{ and that }
	\Ps_{L}^+(p^{(0)}) = \left\{\mat{c}{0 \\ 1 \\ 0 \\ 0 \\ 0},  \mat{c}{0 \\ 0 \\ 1 \\ 0 \\ 0}, \mat{c}{0 \\ 1 \\ 0 \\ 0 \\ 1}\right\}
\end{equation*}
is the corresponding set of postadmissible paths for $p^{(0)}$ which contains $p^{(1)}$.

\vspace{2ex}

As the key trouble of the quadratic performance tests in Theorem \ref{DI::theo::basic_Lyapunov} and \ref{DI::theo::path_Lyapunov}, 
they can neither be nicely generalized to settings involving systems affected by uncertainties nor used for the design of controllers. This stems from the non-convex dependence of the involved inequalities on the describing matrices of \eqref{DI::eq::sys_standard}.
%

In the case of linear impulsive systems with a continuous-time flow component, it has been shown in \cite{Bri13} that this problem can be circumvented by alternative Lyapunov arguments involving a so-called clock. This improvement does, again, not come for free and one usually faces a larger computational burden since many more decision variables are involved. In the discrete-time case and for sequences $(t_k)_{k\in \N_0}$ with \eqref{DI::eq::RDT}, the clock is defined as
\begin{equation}
	\theta:\N_0 \to \N_0,\ \
	\theta(t) := t - t_k-1 \teq{ for all }t \in [t_k+1, t_{k+1}] \cap \N_0 \teq{ and all }k \in \N_0.
	\label{DI::eq::clock}
\end{equation}
It plays a key role in the following performance test, whose specialization to analyzing stability is the discrete-time counterpart of Theorem 2.2 in \cite{Bri13}. In \cite{XiaTra19} it is even shown that the clock and the path based approach can be combined to obtain a less conservative test in which the describing matrices of \eqref{DI::eq::sys_standard} enter in an affine fashion; this is, however, rather expensive and not repeated here.

\begin{theorem}[Clock Based Quadratic Performance Test]
	\label{DI::theo::clock}
	The impulsive system \eqref{DI::eq::sys_standard} achieves quadratic performance with index $P$ for all $(t_k)_{\N_0}$ satisfying \eqref{DI::eq::RDT} if there exist $X_0, \dots, X_{T_{\max}} \in \S^n$ satisfying
	\begin{subequations}
		\label{DI::theo::eq::clock_Lyapunov_LMIs}
		\begin{equation}
			\tlabel{DI::theo::eq::clock_Lyapunov_LMIsa}{DI::theo::eq::clock_Lyapunov_LMIsb}{DI::theo::eq::clock_Lyapunov_LMIsc}
			X_k \cg 0,\quad
			(\bullet)^\top \mat{cc|c}{X_{k+1} & 0 & \\ 0 & -X_k & \\ \hline && P} \mat{cc}{A & B \\ I & 0 \\ \hline C & D \\ 0 & I} \cl 0
			\teq{ and }
			(\bullet)^\top \mat{cc|c}{X_0 & 0 &\\ 0 & -X_k& \\ \hline && P} \mat{cc}{A_J & B_J \\ I & 0 \\ \hline C_J & D_J \\ 0 & I}\cl 0
		\end{equation}
	\end{subequations}
	for all indices $k$ contained in $[0, T_{\max}] \cap \N_0$, $[0, T_{\max}-1] \cap \N_0$ and $[T_{\min}, T_{\max}] \cap \N_0$, respectively.
	Moreover, if the inequalities \eqref{DI::theo::eq::clock_Lyapunov_LMIs} are feasible, then so are the inequalities \eqref{DI::theo::eq::basic_Lyapunov_LMIs}. Finally, if the performance index $P$ is nonsingular, then the converse of the latter statement holds as well.
\end{theorem}

Most of the detailed discussion and the available extensions for Theorem 2.2 in \cite{Bri13} as provided, e.g., in  \cite{Bri13, Bri17, Hol22} remain valid for Theorem \ref{DI::theo::clock} and are not repeated here.
We only emphasize the following points that are relevant in the context of this paper.

\begin{remark}
	
	\begin{enumerate}[(a)]
		\item If compared to Theorem 2.2 in \cite{Bri13}, Theorem \ref{DI::theo::clock} does not involve differential linear matrix inequalities and, hence, does not require the application of relaxations
		(such as those based on linear splines \cite{AllSha10}) in order to arrive at numerically tractable stability criteria. Instead, the inequalities \eqref{DI::theo::eq::clock_Lyapunov_LMIs} can be solved based on any SDP solver right away.
		\item The specialization of Theorem \ref{DI::theo::clock} to a test for stability is obtained by cancelling the last block columns, rows in \eqref{DI::theo::eq::clock_Lyapunov_LMIsb} and  \eqref{DI::theo::eq::clock_Lyapunov_LMIsc} and by making use of $Q \cge 0$. Explicitly, the resulting stability criteria read as
		\begin{equation*}
			X_k \cg 0,\quad
			(\bullet)^\top \mat{cc}{X_{k+1} & 0 \\ 0 & -X_k}\mat{c}{A \\ I} \cl 0
			\teq{ and }
			(\bullet)^\top \mat{cc}{X_0 & 0 \\ 0 & -X_k}\mat{c}{A_J \\ I} \cl 0
		\end{equation*}
		for all indices $k$ contained in $[0, T_{\max}] \cap \N_0$, $[0, T_{\max}-1] \cap \N_0$ and $[T_{\min}, T_{\max}] \cap \N_0$, respectively.
		\item If the two positivity constraints \eqref{DI::theo::eq::clock_Lyapunov_LMIsa} and $Q \cge 0$ are dropped, one can no longer conclude stability and quadratic performance of the system \eqref{DI::eq::sys_standard} based on Theorem \ref{DI::theo::clock}. However, the feasibility of the remaining inequalities \eqref{DI::theo::eq::clock_Lyapunov_LMIsb} and  \eqref{DI::theo::eq::clock_Lyapunov_LMIsc} still implies that the strict \emph{dissipation inequality}
		\begin{equation*}
			x(k+1)^\top X_{\theta(k+1)}x(k+1) - x(0)^\top X_{\theta(0)}x(0) + \sum_{t=0}^{k}(\bullet)^\top P\mat{c}{e(t) \\ d(t)} \leq - \eps \sum_{t = 0}^k \|d(t)\|^2 \text{ ~~holds for all~~ }k\in \N_0
		\end{equation*}
		and for any trajectory of the impulsive system \eqref{DI::eq::sys_standard}. Such inequalities as introduced by Willems in \cite{Wil72a} play a fundamental role in control and will reappear in the next subsection.
		\item Similarly as shown in \cite{Hol22}, Theorem \ref{DI::theo::clock} can be generalized to a quadratic performance test for impulsive systems affected by various types of uncertainties, such as parametric, dynamic or sector bounded nonlinear ones and uncertain delays. This is achieved by a dedicated and clock-based modification of the IQC theorem from \cite{SchVee18}.
		\item Based on Theorem \ref{DI::theo::clock} and due to the affine dependence of the involved LMIs on the underlying system's describing matrices, one can systematically solve a large number of controller design problems beyond finding stabilizing static state-feedback gains \cite{Hol22}.
	\end{enumerate}
	\label{DI::rema::clock}
\end{remark}

Despite of all the favorable properties of the clock-based approach, it has also some downsides.

\begin{itemize}
	\item IQC techniques rely on multipliers for capturing the behavior of uncertainties affecting the underlying system and any such multiplier typically involves several decision variables \cite{MegRan97}. If compared to standard IQC stability tests, several of the extensions of Theorem \ref{DI::theo::clock} as given in \cite{Hol22} require $T_{\max}$ times the number multipliers and, hence, decision variables involved in the prior test. Especially for systems involving a large number of uncertainties, this can become problematic.
	\item Most synthesis approaches covered in \cite{Hol22} involve controllers that require access to the sequence of impulse instants in a causal fashion, which is a natural consequence of using clock-dependent Lyapunov functions as in the proof of Theorem \ref{DI::theo::clock}. This information can, however, be difficult to provide in practice and designing controllers that don't require such knowledge is challenging.
\end{itemize}

To illustrate the latter issue, think of the problem of designing a non-impulsive static state-feedback controller $K$ for an impulsive open-loop plant based on Theorem \ref{DI::theo::clock}. This is not possible directly since products of the form $X_1(A + B_uK), \dots, X_{T_{\max}}(A + B_uK)$ appear and since the standard congruence transformation with $X_k^{-1}$ does not help.
This trouble can be ``resolved'' by enforcing, e.g., $X_0 = \dots = X_{T_{\max}}$ in Theorem \ref{DI::theo::clock}, but this additional constraint typically introduces severe conservatism.
It is often much more reasonable to introduce and enforce suitable constraints on so-called slack variables as proposed in \cite{OliBer99, EbiPea15}.
This leads to the following result whose proof can be extracted from \cite[Theorem 2.12]{Hol22}.

\begin{theorem}[Clock and Slack Variable Based Quadratic Performance Test]
	\label{DI::theo::clock_sv}
	The impulsive system \eqref{DI::eq::sys_standard} achieves quadratic performance with index $P$ for all $(t_k)_{\N_0}$ satisfying \eqref{DI::eq::RDT} if there exist $X_0, \dots, X_{T_{\max}} \in \S^n$ and $G, G_J \in \R^{n\times n}$ satisfying
	\begin{subequations}
		\label{DI::theo::eq::clockSV_Lyapunov_LMIs}
		\begin{equation}
			X_k \cg 0,
		\end{equation}
		\begin{equation}
			\arraycolsep=2pt
			\dlabel{DI::theo::eq::clockSV_Lyapunov_LMIsb}{DI::theo::eq::clockSV_Lyapunov_LMIsc}
			(\bullet)^\top \! \mat{ccc|c}{0 & G & 0 \\ G^\top  & X_{k+1}\!-\! G\!-\!G^\top  & 0 \\ 0 & 0 & -X_k   \\\hline &&& P}\hspace{-1ex}\mat{ccc}{0 & A & B \\ I & 0 & 0 \\ 0 & I & 0 \\  \hline 0 & C & D \\ 0 & 0 & I} \!\cl\! 0
			\text{ and }
			(\bullet)^\top\!  \mat{ccc|c}{0 & G_J & 0 \\ G_J^\top  & X_0\!-\! G_J\!-\!G_J^\top  & 0 \\ 0 & 0 & -X_k   \\\hline &&& P}\hspace{-1ex}\mat{ccc}{0 & A_J & B_J \\ I & 0 & 0 \\ 0 & I & 0 \\  \hline 0 & C_J & D_J \\ 0 & 0 & I} \!\cl\! 0
		\end{equation}
	\end{subequations}
	for all indices $k$ contained in $[0, T_{\max}] \cap \N_0$, $[0, T_{\max}-1] \cap \N_0$ and $[T_{\min}, T_{\max}] \cap \N_0$, respectively.
\end{theorem}

Let us stress that the latter performance test is in general more conservative than the one in Theorem~\ref{DI::theo::clock} and identical to this test if we allow the matrices $G$, $G_J$ to depend on the index $k$.


In the sequel, we analyze the interconnection \eqref{DI::eq::sys} by IQC techniques and by viewing the impulsive operator as an uncertainty. We will  show that this approach permits us to avoid using the clock at some critical spots, which can lead to a reduced computational burden.

\subsection{A Variant of the IQC Theorem}
\label{DI::sec::Recap_IQCs}


\begin{figure}
\begin{minipage}{0.49\textwidth}
	\begin{center}
		\includegraphics{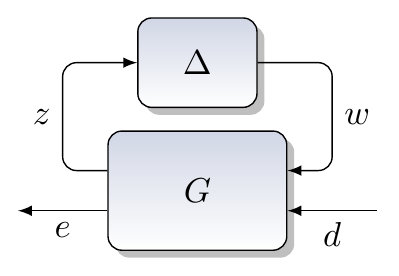}
	\end{center}
\end{minipage}
\begin{minipage}{0.49\textwidth}
	\begin{center}
		\includegraphics{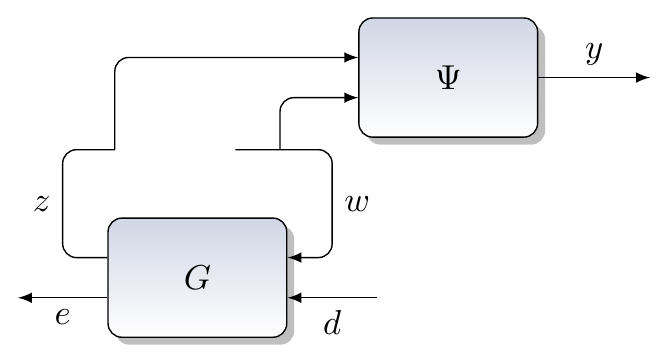}
	\end{center}
\end{minipage}
\caption{Left: Block diagram of an uncertain feedback interconnection \eqref{DI::eq::sys_intercon} with known linear part $G$ and uncertain part $\Delta$. Right: Block diagramm of the system \eqref{DI::eq::sys_augmented} which is the interconnection's linear part $G$ augmented by the filter $\Psi$ in \eqref{DI::eq::filter}.}
\label{DI::fig::block_unc}
\end{figure}

Our quadratic performance tests are based on a variation of the IQC theorem in \cite{SchVee18,Sch21}, which even applies to interconnections \eqref{DI::eq::sys_intercon} involving much more general uncertainties $\Del:\ell_{2e}^{n_z} \to \ell_{2e}^{n_w}$ than the impulsive operator in \eqref{DI::eq::sys::del}. A block diagram of such an interconnection is depicted on the left in Fig.~\ref{DI::fig::block_unc}, where $G$ denotes the interconnection's known linear part.
As the classical variants (see e.g. \cite{MegRan97,VeeSch16}), this IQC theorem involves dynamic multipliers that are described as $\Pi = \Psi^\ast M \Psi$ with a fixed stable dynamic outer factor $\Psi \in \rhi^{m\times (n_z+n_w)}$ and a real middle matrix $M \in \S^m$ which will serve as a variable and is subject to suitable constraints.
In any IQC theorem, the general idea is to find exploitable quadratic constraints expressed in terms of such multipliers $\Pi$, which are enforced on the interconnection signals $z$ and $w$ through the uncertain operator $\Del$.
%

In order to state the IQC theorem, we require some state-space description of the filter $\Psi$ and of the augmented (or filtered) system as depicted on the right in Fig.~\ref{DI::fig::block_unc}. To this end, let us suppose that the output $y$ of $\Psi$ in response to the input $u\in \ell_{2e}^{n_z+n_w}$ is given by
\begin{equation}
\mat{c}{\xi(t+1) \\ y(t)} = \mat{cc}{A_\Psi & B_\Psi \\ C_\Psi & D_\Psi} \mat{c}{\xi(t) \\ u(t)}
\label{DI::eq::filter}
\end{equation}
for $t\in \N_0$ and with a zero initial condition, i.e., $\xi(0) = 0 \in \R^{n_\xi}$. Then
the augmented system in Fig.~\ref{DI::fig::block_unc} admits the description
\begin{equation}
\label{DI::eq::sys_augmented}
\mat{c}{\xi(t+1) \\ x(t+1) \\ \hline y(t)\\ \hdashline e(t)} = \mat{cc|c:c}{A_\Psi & B_\Psi \smat{C_z \\ 0} & B_\Psi \smat{D_{zw} \\ I_{n_w}} & B_\Psi \smat{D_{zd} \\ 0}\\ 0 & A & B_w & B \\\hline C_\Psi & D_\Psi \smat{C_z \\ 0} & D_\Psi \smat{D_{zw} \\ I_{n_w}} & D_\Psi \smat{D_{zd} \\ 0} \\ \hdashline
	0 & C & D_{ew} & D}
\mat{c}{\xi(t) \\ x(t) \\ \hline w(t) \\ \hdashline d(t)}
\end{equation}
for $t\in \N_0$. Our variant of the IQC theorem reads as follows and is proved in the appendix.

\begin{theorem}[IQC Theorem]
\label{DI::theo::IQC}
Suppose that the interconnection \eqref{DI::eq::sys_intercon} is well-posed and that $\Del:\ell_{2e}^{n_z} \to \ell_{2e}^{n_w}$ satisfies an \emph{IQC with terminal cost} with respect to the filter \eqref{DI::eq::filter}, the matrix $M \in \S^m$, the map $Z:\N \to \S^{n_\xi}$, and the strictly increasing sequence of integers $(t_k)_{k \in \N_0}$ with $t_0 = -1$. The latter means that the inequality
\begin{equation}
	\xi(t_{k}+1)^\top Z(k)\xi(t_k+1)  + \sum_{t = 0}^{t_k}y(t)^\top M y(t) \geq 0
	\teq{ for all }k \in \N
	\tag{IQC}
	\label{DI::eq::pointwise_IQC}
\end{equation}
holds for any state and output trajectory of the filter \eqref{DI::eq::filter} driven by the input $u = \smat{z \\ \Del(z)}$ with any $z \in \ell_{2e}^{n_z}$.
Then the interconnection \eqref{DI::eq::sys_intercon} is stable in the sense that its state satisfies $x \in \ell_2^n$ for any initial condition $x(0) \in \R^n$ and it achieves quadratic performance with index $P$ if there exists a matrix $X \in \S^{n_\xi+n}$ satisfying
\begin{subequations}
	\label{DI::theo::eq::main}
	\begin{equation}
		\dlabel{DI::theo::eq::maina}{DI::theo::eq::mainb}
		(\bullet)^\top \mat{cc|c:c}{X & 0 & &\\ 0 & -X & &\\ \hline &&  M& \\ \hdashline &&& P}
		\mat{cc|c:c}{A_\Psi & B_\Psi \smat{C_z \\ 0} & B_\Psi \smat{D_{zw} \\ I_{n_w}} & B_\Psi \smat{D_{zd} \\ 0} \\ 0 & A & B_w & B \\ I& 0 & 0 & 0\\ 0& I & 0 & 0\\ \hline C_\Psi & D_\Psi \smat{C_z \\ 0} & D_\Psi \smat{D_{zw} \\ I_{n_w}} & D_\Psi \smat{D_{zd} \\ 0}\\ \hdashline 0 & C & D_{ew} & D \\ 0 & 0 & 0 & I} \cl 0
		\teq{ and }
		X - \mat{cc}{Z(k) & 0 \\ 0 & 0}\cg 0
	\end{equation}
\end{subequations}
for all $k \in \N$.
\end{theorem}

Let us emphasize that this is rather straightforward 
modification of the continuous-time IQC theorem proposed in \cite{SchVee18,Sch21}. Indeed, the discrete-time version of these previous results are recovered by restricting $Z$ to constant and by choosing $t_k := k-1$ for all $k \in \N_0$.
Moreover, let us stress that a detailed discussion involving an analysis of links with classical IQCs as proposed in \cite{MegRan97} and the corresponding stability tests can be found in these papers as well.
We only recall that, by $\xi(0) = 0$, the inequality \eqref{DI::eq::pointwise_IQC} constitutes a discrete-time non-strict \emph{dissipation inequality} \cite{Wil72a}.
Here, the matrix $M$ defines the supply rate and $Z$ characterizes a partial storage function that only involves the filter's states but does not rely on any information about ``internal" properties or quantities of the uncertainty $\Delta$. The term $\xi(t_k+1)^\top Z(k) \xi(t_k+1)$ can also be viewed as a terminal cost which justifies the naming.

In contrast to the notion of IQCs in \cite{Hol22}, we choose constant matrices $X$ and $M$ such that Theorem~\ref{DI::theo::IQC} is closer to the IQC theorem proposed in \cite{SchVee18,Sch21}. One could as well pick $k$-dependent matrices which might lead to less conservatism, but the improvements are expected to be disproportionate to the resulting increased computational burden. 

Finally, note that Theorem \ref{DI::theo::IQC} involves a rather weak notion of stability, namely global attractivity, because it only requires very mild assumptions on the considered operator $\Del$. This permits its application even if $\Del$ consists of an impulsive operator \eqref{DI::eq::sys::del} and, at the same time, of numerous other types of uncertainties as mentioned earlier.
If $\Del$ is identical to the impulsive operator in \eqref{DI::eq::sys::del}, then one can recover exponential stability as defined in Definition \ref{DI::def::stab} in all of the upcoming tests. Indeed, this is seen with some minor changes of the arguments and by showing that the impulsive operator even satisfies,
for some $\t \rho \in (0, 1)$,
the somewhat stronger dissipation-like inequality
\begin{equation*}
(\bullet)^\top \mat{cc}{Z(k+1) & 0 \\ 0 & -\rho Z(k)}\mat{c}{\xi(t_{k+1}+1) \\ \xi(t_k+1)} + \sum_{t = t_k+1}^{t_{k+1}}y(t)^\top My(t) \geq 0
\quad\text{ for all }k \in \N_0
\text{ and all }\rho \in (\t \rho, 1).
\end{equation*}



\subsection{New IQC Based Stability Tests for the Interconnection \eqref{DI::eq::sys}} \label{DI::sec::main_ana}

After these preparation, we are ready to formulate
the main contribution of this paper's analysis part, namely the construction of suitable IQCs with terminal cost for the particular impulsive operator $\Del$ in \eqref{DI::eq::sys::del}. This serves
the purpose of deriving new quadratic performance tests for the interconnection \eqref{DI::eq::sys}
based on the IQC theorem \ref{DI::theo::IQC}.
%
%
%
%
To this end, we begin by providing an IQC for such an impulsive operator corresponding to a sequence of impulse instants with \eqref{DI::eq::ADT}. Since, in this case, there is only limited information that can be taken into account, this will yield the simplest quadratic performance test. Our first result is mainly motivated by its straightforward proof and involves a trivial static filter $\Psi$.

\begin{lemma}
\label{DI::lem::pIQC_ADT}
Let $(t_k)_{k\in \N_0}$ be a sequence of impulse instants with \eqref{DI::eq::ADT}. Then the corresponding operator $\Del$ in \eqref{DI::eq::sys::del} satisfies \eqref{DI::eq::pointwise_IQC} for the static filter $\Psi := I_{n_z+n_w}$, for the sequence $(t_k)_{k \in \N_0}$ and for any matrix $M \in \S^{n_z+n_w}$ with
\begin{equation}
	\mat{c}{I \\ 0}^\top M \mat{c}{I \\ 0} \cge 0 \teq{ and }
	\mat{c}{I \\ I}^\top M \mat{c}{I \\ I} \cge 0.
	\label{DI::lem::eq::static_multiplier}
\end{equation}
\end{lemma}

\begin{proof}
Since the filter $\Psi = I_{n_z+n_w}$ is static, there is no state $\xi$, i.e., $n_\xi = 0$. Moreover, the mere definition of $\Del$ in \eqref{DI::eq::sys::del} implies
\begin{equation*}
	y(t)^\top My(t) = (\bullet)^\top M \mat{c}{z(t) \\ \Del(z)(t)} = z(t)^\top \mat{c}{I \\ 0}^\top M \mat{c}{I \\ 0}z(t) \stackrel{\eqref{DI::lem::eq::static_multiplier}}{\geq} 0
	\teq{ for all }t \in \N_0 \setminus \{t_1, t_2, \dots \}
\end{equation*}
and
\begin{equation*}
	y(t_k)^\top My(t_k) = (\bullet)^\top M \mat{c}{z(t_k) \\ \Del(z)(t_k)} = z(t_k)^\top\mat{c}{I \\ I}^\top M \mat{c}{I \\ I}z(t_k) \stackrel{\eqref{DI::lem::eq::static_multiplier}}{\geq} 0
	\teq{ for all }k\in \N.
\end{equation*}
It remains to observe that these two inequalities yield a special case of \eqref{DI::eq::pointwise_IQC} which completes the proof.
\end{proof}

As a consequence of Theorem \ref{DI::theo::IQC} and Lemma \ref{DI::lem::pIQC_ADT}, we obtain the following quadratic performance test.

\begin{corollary}
\label{DI::coro::ana_ADT}
Let $\Psi := I_{n_z+n_w}$. The interconnection \eqref{DI::eq::sys} achieves quadratic performance with index $P$ for all $(t_k)_{k\in \N_0}$ with \eqref{DI::eq::ADT} if there exist matrices $X\in \S^n$ and $M\in \S^{n_z+n_w}$ satisfying \eqref{DI::theo::eq::main} and \eqref{DI::lem::eq::static_multiplier}.
\end{corollary}

\begin{proof}

\emph{Well-Posedness:} 
Note that the $(2, 2)$ block of \eqref{DI::theo::eq::maina} in the indicated partition and \eqref{DI::theo::eq::mainb} read as $B_w^\top XB_w + (\bullet)^\top M \smat{D_{zw} \\ I} +(\bullet)^\top P \smat{D_{ew} \\ 0} \cl 0$ and $X \cg 0$, respectively. This implies $(\bullet)^\top M \smat{D_{zw} \\ I} \cl 0$ by the assumption on the left upper block of the performance index $P$. Together with $(\bullet)^\top M \smat{I \\ I} \cge 0$, we can conclude that $I - D_{zw}$ is nonsingular, i.e., the well-posedness of the interconnection \eqref{DI::eq::sys}. 

\emph{Quadratic Performance:}    Let $(t_k)_{k\in \N_0}$ with \eqref{DI::eq::ADT} be arbitrary. By assumption and Lemma \ref{DI::lem::pIQC_ADT}, the corresponding operator $\Del$ in \eqref{DI::eq::sys::del} satisfies \eqref{DI::eq::pointwise_IQC} for the given matrix $M$, an empty matrix $Z$ and for the sequence $(t_k)_{k \in \N_0}$. As we have shown well-posedness of the interconnection \eqref{DI::eq::sys} and since \eqref{DI::theo::eq::main} is satisfied for the given matrices $M$ and $X$, we can apply Theorem \ref{DI::theo::IQC} to conclude that the interconnection achieves quadratic performance with index $P$. Since $(t_k)_{k\in \N_0}$ with \eqref{DI::eq::ADT} was arbitrary, we have shown the claim.
\end{proof}


The construction of suitable IQCs is much more interesting and challenging if the dwell-time of the considered impulse sequences is not arbitrary. Such additional knowledge
permits or even requires the use of nontrivial filters $\Psi$. In the sequel, we focus on impulse sequences with \eqref{DI::eq::RDT} which include those with \eqref{DI::eq::EDT} and only comment on those with \eqref{DI::eq::MDT}.
For implementations, we can choose filters given by the transfer matrix
\begin{subequations}
\label{DI::eq::filter_TF_SS}
\begin{equation}
	\Psi(z) := I_{n_z+n_w} \kron \mat{cccc}{\frac{1}{z^\nu} & \dots & \frac{1}{z} & 1}^\top,
\end{equation}
for example, with some length $\nu \in \N_0$. Note that $\Psi$ admits the realization
\begin{equation}
	(I_{n_z+n_w} \kron J_\nu, ~ I_{n_z+n_w} \kron e_\nu, ~ I_{n_z+n_w} \kron C_\nu, ~ I_{n_z+n_w} \kron e_{\nu+1}).
\end{equation}
\end{subequations}
Here, $e_\nu := \smat{0_{\nu-1} \\1}$, $C_\nu := \smat{I_\nu \\ 0_{1\times \nu}}$, $e_{\nu+1} := \smat{0_\nu \\ 1}$ and $J_\nu \in \R^{\nu \times \nu}$ is an (upper) Jordan-block with eigenvalue zero. This common choice is motivated by its nice approximation properties \cite{SchKoe12, Pin85}.
The length $\nu$ can be viewed as a tuning knob for trading off conservatism versus computational burden. In many practical situations, small positive lengths (say smaller than $4$) are sufficient to achieve good results.

\vspace{1ex}

In order to construct IQCs with nontrivial filters $\Psi$, recall that this essentially amounts to finding suitable dissipation inequalities involving the output $y$ and state $\xi$ of the filter $\Psi$ in response to the input $\smat{z \\ \Del(z)}$, i.e., dissipation inequalities for the auxiliary system depicted in Fig.~\ref{DI::fig::block_unc2}.
Now, it is a simple but instrumental observation that, due to the definition of the
impulsive operator $\Del$ in \eqref{DI::eq::sys::del}, this auxiliary system can be explicitly described in state-space by
\begin{equation}
\label{DI::eq::sys_auxiliary}
\mat{c}{\xi(t+1) \\ y(t)} = \mat{cc}{A_\Psi & B_\Psi \smat{I_{n_z} \\ 0} \\[1ex] C_\Psi & D_\Psi \smat{I_{n_z} \\ 0}} \mat{c}{\xi(t) \\ z(t)},\quad
\mat{c}{\xi(t_k+1) \\ y(t_k)} = \mat{cc}{A_\Psi & B_\Psi \smat{I_{n_z} \\ I_{n_z}} \\[1.5ex] C_\Psi & D_\Psi \smat{I_{n_z} \\ I_{n_z}}} \mat{c}{\xi(t_k) \\ z(t_k)}
\end{equation}
for $t \in \N_0 \setminus \{t_1, t_2, \dots \}$ and $k \in \N$. This is is a standard linear impulsive system similar to the one in \eqref{DI::eq::sys_standard}.
This insight enables us to construct IQCs based on employing the recapitulated results in Section~\ref{DI::sec::recap}. We just have to replace the tests for quadratic performance with index $P$ by the corresponding tests for achieving non-strict dissipativity with supply rate $M$, as indicated in Remark \ref{DI::rema::clock} (c) for the clock based approach. This leads to the following result.

\begin{figure}
\begin{center}
	\includegraphics{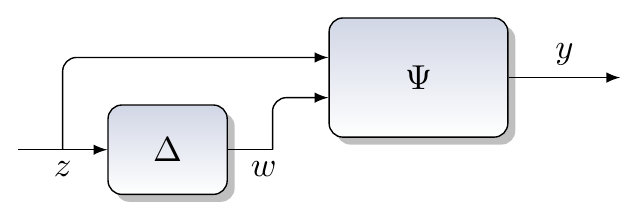}
\end{center}
\caption{Block diagram of the auxiliary system \eqref{DI::eq::sys_auxiliary} appearing in \eqref{DI::eq::pointwise_IQC} involving the impulsive operator \eqref{DI::eq::sys::del} and the filter~\eqref{DI::eq::filter}.}
\label{DI::fig::block_unc2}
\end{figure}

\subsubsection{Clock Based IQCs with Terminal Cost for the Impulse Operator}

\begin{lemma}
\label{DI::lem::pIQC_RDT}
Let $(t_k)_{k \in \N_0}$ be a sequence of impulse instants with \eqref{DI::eq::RDT}. Then the corresponding operator $\Del$ in \eqref{DI::eq::sys::del} satisfies \eqref{DI::eq::pointwise_IQC} for the sequence $(t_k)_{k\in \N_0}$ and for any matrices $M \in \S^{m}$ and $Z_0, \dots, Z_{T_{\max}} \in \S^{n_\xi}$ satisfying
\begin{subequations}
	\label{DI::lem::eq::dynamic_multiplier_RDT}
	\begin{equation}
		\dlabel{DI::lem::eq::dynamic_multiplier_RDTa}{DI::lem::eq::dynamic_multiplier_RDTb}
		(\bullet)^\top \mat{cc|c}{Z_{k+1} & 0 & \\ 0 & -Z_k & \\ \hline  && M}
		\mat{cc}{A_\Psi & B_\Psi \smat{I_{n_z} \\ 0} \\ I & 0 \\ \hline C_\Psi & D_\Psi\smat{I_{n_z} \\ 0}} \cge 0
		\teq{and}
		(\bullet)^\top \mat{cc|c}{Z_{0} & 0 & \\ 0 & -Z_k & \\ \hline  && M}
		\mat{cc}{A_\Psi & B_\Psi \smat{I_{n_z} \\ I_{n_z}} \\ I & 0 \\ \hline C_\Psi & D_\Psi\smat{I_{n_z} \\ I_{n_z}}} \cge 0
	\end{equation}
\end{subequations}
for all indices $k$ contained in $[0, T_{\max} - 1]\cap \N_0$ and $[T_{\min}, T_{\max}] \cap \N_0$, respectively.
\end{lemma}

Together with Theorem \ref{DI::theo::IQC}, we obtain our first novel IQC based quadratic performance test for the interconnection \eqref{DI::eq::sys} and sequences of impulse instants with \eqref{DI::eq::RDT} involving a nontrivial dynamic filter $\Psi$. A new result for sequences satisfying \eqref{DI::eq::MDT} is obtained by setting $T_{\max} := T_{\min}$ and including the LMI
\begin{equation*}
(\bullet)^\top \mat{cc|c}{Z_{T_{\min}} & 0 & \\ 0 & -Z_{T_{\min}} & \\ \hline  && M}
\mat{cc}{A_\Psi & B_\Psi \smat{I_{n_z} \\ 0} \\ I & 0 \\ \hline C_\Psi & D_\Psi\smat{I_{n_z} \\ 0}} \cge 0.
\end{equation*}
Note that the proof of the latter variant relies on modifying the clock \eqref{DI::eq::clock} as
\begin{equation*}
%
\theta(t) := t- t_k-1 \text{ for }t \in [t_k+1, t_k+1+T_{\min}]\cap \N_0
\teq{and}
\theta(t) := T_{\min} \text{ for }t \in (t_k+1+T_{\min}, t_{k+1}] \cap \N_0.
%
\end{equation*}

\begin{theorem}[IQC and Clock Based Quadratic Performance Test]
\label{DI::coro::ana_RDT}
The interconnection \eqref{DI::eq::sys} achieves quadratic performance with index $P$ for all $(t_k)_{k\in \N_0}$ with \eqref{DI::eq::RDT} if there exist matrices $X \in \S^{n_\xi + n}$, $M\in \S^{m}$ and $Z_0, \dots, Z_{T_{\max}}\in \S^{n_\xi}$ satisfying \eqref{DI::theo::eq::maina}, \eqref{DI::lem::eq::dynamic_multiplier_RDT} and
\begin{equation*}
	X - \mat{cc}{Z_k & 0 \\ 0 & 0 }\cg 0
	\teq{ for all }k \in [0, T_{\max}]\cap \N_0.
\end{equation*}
\end{theorem}

We emphasize that, in contrast to other IQC based stability or performance tests, Theorem \ref{DI::coro::ana_RDT} does not require well-posedness of the interconnection \eqref{DI::eq::sys} as an additional assumption. Instead, well-posedness can be concluded in a natural fashion once the involved LMIs are feasible.

\begin{proof}
As illustrated in the proof of Corollary \ref{DI::coro::ana_ADT}, it suffices to show that feasibility of the latter LMIs implies well-posedness of the interconnection \eqref{DI::eq::sys} in order to apply Theorem \ref{DI::theo::IQC} and Lemma \ref{DI::lem::pIQC_RDT}. To this end, observe that the right lower block of \eqref{DI::lem::eq::dynamic_multiplier_RDTb} yields
\begin{equation*}
	\mat{c}{I\\ I}^\top P \mat{c}{I\\ I} \cge 0
	\teq{ for } P := \mat{c}{B_\Psi \\ D_\Psi}^\top \mat{cc}{Z_0 & 0 \\ 0 & M}\mat{c}{B_\Psi \\ D_\Psi}.
\end{equation*}
On the other hand, $X - \smat{Z_k & 0 \\ 0 & 0} \cg 0$ for $k = 0$ implies
\begin{equation*}
	(\bullet)^\top P \mat{c}{D_{zw} \\ I}
	\cle (\bullet)^\top \left(X - \mat{cc}{Z_0 & 0 \\ 0 & 0}\right) \mat{c}{B_\Psi \smat{D_{zw} \\ I} \\ B_w} + (\bullet)^\top  \mat{c}{B_\Psi \\ D_\Psi}^\top \mat{cc}{Z_0 & 0 \\ 0 & M}\mat{c}{B_\Psi \\ D_\Psi} \mat{c}{D_{zw} \\ I}
\end{equation*}
and the right hand side can be simplified to
\begin{equation*}
	(\bullet)^\top X \mat{c}{B_\Psi \smat{D_{zw} \\ I} \\ B_w} + (\bullet)^\top M D_\Psi \smat{D_{zw} \\ I}
	\stackrel{Q \cge 0}{\cle} 	(\bullet)^\top X \mat{c}{B_\Psi \smat{D_{zw} \\ I} \\ B_w} + (\bullet)^\top M D_\Psi \smat{D_{zw} \\ I} + (\bullet)^\top P \mat{c}{D_{ew} \\ 0} \cl 0.
\end{equation*}
Here, the latter inequality follows from considering the $(2,2)$ block of \eqref{DI::theo::eq::maina} in the indicated partition. In summary, we have shown $(\bullet)^\top P \smat{I\\ I} \cge 0$ and $(\bullet)^\top P \smat{D_{zw} \\ I} \cl 0$ which implies that $I - D_{zw}$ is nonsingular and, hence, that the interconnection \eqref{DI::eq::sys} is well-posed.
\end{proof}

We can relate the above quadratic performance test to the clock-based one in Theorem \ref{DI::theo::clock} as 
proposed by \cite{Bri13}. This is somewhat surprising since the latter can be viewed as resulting from a parameter-dependent Lyapunov function approach and because it is in general difficult to (theoretically) compare such approaches with those based on IQCs.

\begin{lemma}
\label{DI::lem::conservativsm_estimate_clock}
Let $\smat{A_J & B_J \\ C_J & D_J} := \smat{A & B \\ C & D} + \smat{B_w \\ D_{ew}}(I - D_{zw})^{-1}(C_z, D_{zd})$  and let the performance index $P$ be nonsingular. Moreover, suppose that the inequalities in Theorem \ref{DI::coro::ana_RDT} are feasible. Then the inequalities \eqref{DI::theo::eq::clock_Lyapunov_LMIs} in Theorem \ref{DI::theo::clock} are feasible as well.
\end{lemma}

In other words, the clock based IQC performance test in Theorem \ref{DI::coro::ana_RDT} could be more
conservative than the purely clock based one in Theorem \ref{DI::theo::clock}. Based on our observations for several numerical examples, we conjecture that taking the filter $\Psi$ as in \eqref{DI::eq::filter_TF_SS} and letting the length $\nu$ approach infinity yields asymptotically the same conservatism.

\begin{proof}
Let $k \in [0, T_{\max}-1]\cap \N_0$ be fixed. Then we can introduce the matrix
\begin{equation}\label{h0}
	K_k := (\bullet)^\top \mat{cc|c}{Z_{k+1} & 0 & \\ 0 & -Z_k & \\ \hline && M}\mat{cc}{A_\Psi & B_\Psi \\ I & 0 \\ \hline C_\Psi & D_\Psi}
	\teq{ which satisfies }
	(\bullet)^\top K_k \mat{cc}{I & 0 \\ 0 & I_{n_z} \\ 0 & 0} \cge 0
\end{equation}
due to \eqref{DI::lem::eq::dynamic_multiplier_RDTa}. By elementary computations, we can incorporate this matrix into \eqref{DI::theo::eq::maina} which yields
\begin{equation*}
	(\bullet)^\top \mat{cc|c:c}{\t X_{k+1} & 0 & &\\ 0 & -\t X_k & &\\ \hline &&  K_k& \\ \hdashline &&& P}
	\mat{cc|c:c}{A_\Psi & B_\Psi \smat{C_z \\ 0} & B_\Psi \smat{D_{zw} \\ I_{n_w}} & B_\Psi \smat{D_{zd} \\ 0} \\ 0 & A & B_w & B \\ I& 0 & 0 & 0\\ 0& I & 0 & 0\\ \hline \smat{I \\ 0 \\ 0} & \smat{0 \\ C_z \\ 0} & \smat{0 \\ D_{zw} \\ I_{n_w}} &  \smat{0 \\ D_{zd} \\ 0}\\ \hdashline 0 & C & D_{ew} & D \\ 0 & 0 & 0 & I} \cl 0
\end{equation*}
for $\t X_k := X - \smat{Z_k & 0 \\ 0 & 0}$ and $\t X_{k+1} := X - \smat{Z_{k+1} & 0 \\ 0 & 0}$. These matrices are positive definite by assumption. Canceling the second block row and column allows us to make use of \eqref{h0}, 
which leads to
\begin{equation*}
	(\bullet)^\top \mat{cc|c}{\t X_{k+1} & 0 & \\ 0 & -\t X_k & \\ \hline  && P}
	\mat{cc|c}{A_\Psi & B_\Psi \smat{C_z \\ 0}  & B_\Psi \smat{D_{zd} \\ 0} \\ 0 & A & B \\ I& 0  & 0\\ 0& I  & 0\\ \hline  0 & C &  D \\ 0 & 0 & I} \cl 0.
\end{equation*}
Since $P$, $\t X_k$ and $\t X_{k+1}$ are nonsingular, we can apply the dualization lemma (see, e.g., \cite{SchWei00})  in order to infer
\begin{equation*}
	(\bullet)^\top \mat{cc|c}{\t X_{k+1}^{-1} & 0 & \\ 0 & -\t X_k^{-1} & \\ \hline  && P^{-1}}
	\mat{ccc}{I & 0 & 0  \\ 0 & I & 0 \\ -A_\Psi^\top  & 0 & 0 \\ -\left(B_\Psi \smat{C_z \\ 0}\right)^\top & -A^\top & -C^\top \\ \hline  0 & 0 & I \\ -\left(B_\Psi \smat{D_{zd} \\ 0}\right)^\top & -B^\top & -D^\top } \cg 0.
\end{equation*}
Canceling the first block column then results in
\begin{equation*}
	(\bullet)^\top \mat{cc|c}{X_{k+1}^{-1} & 0 & \\ 0 & -X_k^{-1} & \\ \hline  && P^{-1}}
	\mat{cc}{I & 0 \\-A^\top & -C^\top \\ \hline  0 & I \\ -B^\top & -D^\top } \cg 0
\end{equation*}
for $X_k := \left(\smat{0 \\ I}^\top \t X_{k}^{-1}\smat{0 \\ I}\right)^{-1}$ and $X_{k+1}:= \left(\smat{0 \\ I}^\top \t X_{k+1}^{-1}\smat{0 \\ I}\right)^{-1}$ which are both positive definite. Since $k \in [0, T_{\max}-1]\cap \N_0$ was arbitrary and by applying the dualization lemma once more, we can conclude that \eqref{DI::theo::eq::clock_Lyapunov_LMIsa} and \eqref{DI::theo::eq::clock_Lyapunov_LMIsb} are satisfied. In a similar fashion, we can conclude that \eqref{DI::theo::eq::clock_Lyapunov_LMIsc} holds for the same matrices $X_k$ based on \eqref{DI::lem::eq::dynamic_multiplier_RDTb}.
\end{proof}

\subsection{Lifting Based IQCs with Terminal Cost for the Impulse Operator}

We obtain another IQC for the impulse operator if modifying the quadratic performance test based on the discrete-time lifting technique in Theorem \ref{DI::theo::basic_Lyapunov} in the same fashion as we did for the clock-based approach. This leads to the following result.

\begin{lemma}
\label{DI::lem::pIQC_RDT2}
Let $(t_k)_{k \in \N_0}$ be a sequence of impulse instants with \eqref{DI::eq::RDT}. Then the corresponding operator $\Del$ in \eqref{DI::eq::sys::del} satisfies \eqref{DI::eq::pointwise_IQC} for the sequence $(t_k)_{k\in \N_0}$ and for any matrices $M \in \S^{m}$ and $Z\in \S^{n_\xi}$ with
\begin{equation}
	\label{DI::lem::eq::dynamic_multiplier_RDT2}
	(\bullet)^\top \mat{cc|c}{Z & 0 & \\ 0 & -Z & \\ \hline  && I_{k+1} \kron M}
	\mat{cccccc}{A_\Psi^{k+1} & A_\Psi^{k}B_\Psi\smat{I \\ 0} & A_\Psi^{k-1}B_\Psi \smat{I\\ 0} & \dots & A_\Psi B_\Psi \smat{I \\ 0} & B_\Psi \smat{I \\ I}
		\\ I & 0 & \dots & \dots & \dots & 0\\ \hline C_\Psi & D_\Psi\smat{I \\ 0}& 0 & \dots & \dots & 0 \\
		C_\Psi A_\Psi & C_\Psi B_\Psi \smat{I \\ 0} & D_\Psi \smat{I\\ 0} &  &  & \vdots \\
		\vdots & \vdots & \vdots &  &  & 0 \\
		C_\Psi A_\Psi^{k} & C_\Psi A_\Psi^{k-1} B_\Psi \smat{I \\ 0} & C_\Psi A_\Psi^{k-2} B_\Psi \smat{I\\ 0} & \dots & C_\Psi B_\Psi \smat{I\\ 0} & D_\Psi \smat{I \\ I}} \cge 0
\end{equation}
for all $k\in [T_{\min}, T_{\max}]\cap \N$.
\end{lemma}


Together with Theorem \ref{DI::theo::IQC}, we obtain another novel IQC based quadratic performance test for the interconnection \eqref{DI::eq::sys} and sequences of impulse instants with \eqref{DI::eq::RDT}. The corresponding test for impulse sequences satisfying \eqref{DI::eq::MDT} is obtained by setting $T_{\max} := T_{\min}$ and including the LMI
\begin{equation*}
(\bullet)^\top \mat{cc|c}{Z & 0 & \\ 0 & -Z & \\ \hline  && M}
\mat{cc}{A_\Psi & B_\Psi \smat{I_{n_z} \\ 0} \\ I & 0 \\ \hline C_\Psi & D_\Psi\smat{I_{n_z} \\ 0}} \cge 0.
\end{equation*}

\begin{theorem}[IQC and Lifting Based Quadratic Performance Test]
\label{DI::coro::ana_RDT2}
The interconnection \eqref{DI::eq::sys} achieves quadratic performance with index $P$ for all $(t_k)_{k\in \N_0}$ with \eqref{DI::eq::RDT} if there exist matrices $X \in \S^{n_\xi + n}$, $M\in \S^{m}$ and $Z\in \S^{n_\xi}$ satisfying \eqref{DI::theo::eq::maina}, \eqref{DI::lem::eq::dynamic_multiplier_RDT2} and
\begin{equation*}
	X - \mat{cc}{Z & 0 \\ 0 & 0 }\cg 0.
\end{equation*}
Moreover, if the inequalities in Theorem \ref{DI::coro::ana_RDT} are feasible, then so are the above inequalities.
\end{theorem}

Note that the second statement can be shown similarly as the second statement in Theorem \ref{DI::theo::clock}. Hence, the IQC and lifting based performance test is guaranteed to be no more conservative than the IQC and clock based test. However and in contrast to the relation of the original lifting and clock based tests, the converse of this statement is not true. As one reason, the matrices $Z_0, \dots, Z_{T_{\max}}$ are in general not positive definite, which would be required when following the proof of the third statement in Theorem \ref{DI::theo::clock}.
The numerical examples provided in the next subsection illustrates the 
gap between the tests in Theorems \ref{DI::coro::ana_RDT} and \ref{DI::coro::ana_RDT2}.

\begin{proof}[Proof of first statement.]
As illustrated in the proof of Corollary \ref{DI::coro::ana_ADT}, it suffices to show that feasibility of the LMIs in Theorem \ref{DI::coro::ana_RDT2} implies well-posedness of the interconnection \eqref{DI::eq::sys} in order to apply Theorem \ref{DI::theo::IQC} and Lemma \ref{DI::lem::pIQC_RDT2}.
To this end, note that canceling the first block row and column of \eqref{DI::lem::eq::dynamic_multiplier_RDT2} yields
\begin{equation*}
	\arraycolsep=1pt
	(\bullet)^\top P \!\underbrace{\mat{cccc}{\smat{I\\ 0} & & & \\ & \ddots  & & \\ && \smat{I \\ 0} & \\ & & & \smat{I \\ I}}}_{=: \mathscr{I}} \cge 0
	\teq{ for } P := (\bullet)^\top \mat{cc}{Z & 0 \\ 0 & I_{\sigma+1}  \kron M}\!
	\underbrace{\mat{ccccc}{A_\Psi^{\sigma}B_\Psi & A_\Psi^{\sigma-1}B_\Psi  & \dots & B_\Psi \\
			\hline D_\Psi& 0 & \dots & 0 \\
			C_\Psi B_\Psi  & D_\Psi  & \ddots & \vdots \\
			\vdots & \vdots & \ddots & 0 \\
			C_\Psi A_\Psi^{\sigma-1} B_\Psi& \dots  & \dots & D_\Psi}}_{=: \smat{\h B_\Psi \\ \h D_\Psi}}
\end{equation*}
for $\sigma := T_{\max}$. Next, let us abbreviate the filtered system matrices in \eqref{DI::eq::sys_augmented} by $\Ac$, $\Bc_w$, $\Bc$, $\Cc_z$, $\Cc$, $\Dc_{zw}$, $\Dc_{zd}$, $\Dc_{ew}$, $\Dc$ and let us define the corresponding lifted matrices $\h \Bc_w$ and $\h \Dc_{zw}$ analogously as $\h B_\Psi$ and $\h D_\Psi$, respectively. Then we get, after somewhat tedious but elementary computations,
\begin{equation*}
	\mat{cc|c}{I_{n_\xi} & 0 & 0\\ \hline 0 & 0 & I}\mat{c}{\h \Bc_{w} \\ \hline \h \Dc_{zw}}
	= \mat{c}{\h B_\Psi \\ \hline \h D_\Psi} \underbrace{\mat{cccc}{\smat{D_{zw} \\ I} & 0 & \dots & 0 \\ \smat{C_zB_w \\ 0} & \smat{D_{zw} \\ I} & \ddots & \vdots \\
			\vdots & & \ddots & 0 \\
			\smat{C_zA^{\sigma-2}B_w \\ 0} & \dots & \dots & \smat{D_{zw} \\ I}}}_{=: \mathscr{D}}.
\end{equation*}
By $X - \smat{Z & 0 \\ 0& 0} \cg 0$, we can then conclude
\begin{equation*}
	\mathscr{D}^\top P \mathscr{D}
	\cle \h \Bc_w^\top \left(X \!-\! \mat{cc}{Z & 0 \\ 0 & 0}\right) \h \Bc_w  + \mathscr{D}^\top  \mat{c}{\h B_\Psi \\ \h D_\Psi}^{\!\top} \!\mat{cc}{Z & 0 \\ 0 & I_{\sigma+1} \!\kron\! M}\mat{c}{\h B_\Psi \\ \h D_\Psi} \mathscr{D} \\
	= (\bullet)^\top \!\mat{cc}{X & 0 \\ 0 & I_{\sigma+1} \!\kron\! M} \mat{c}{\h \Bc_w \\ \h \Dc_{zw}}\!.
\end{equation*}
The last term is negative definite because it results from considering the $(2, 2)$ block of the lifted LMI corresponding to \eqref{DI::theo::eq::maina}. In summary, we have $\mathscr{I}^\top P \mathscr{I} \cge 0$ and $\mathscr{D}^\top P \mathscr{D} \cl 0$ which implies $\mathrm{im}(\mathscr{I}) \cap \mathrm{im}(\mathscr{D}) = \{0\}$. By the particular structure of these matrices, one observes that $I - D_{zw}$ is nonsingular and, hence, that the interconnection \eqref{DI::eq::sys} is well-posed.
\end{proof}

As before, we can show that the quadratic performance test in Theorem \ref{DI::coro::ana_RDT2} is in general more conservative than the one given in Theorem \ref{DI::theo::basic_Lyapunov} and we conjecture that this discrepancy vanishes asymptotically if the length of the employed filter approaches infinity.

\begin{lemma}
\label{DI::lem::conservativsm_estimate_lifting}
Let $\smat{A_J & B_J \\ C_J & D_J} := \smat{A & B \\ C & D} + \smat{B_w \\ D_{ew}}(I - D_{zw})^{-1}(C_z, D_{zd})$  and let the performance index $P$ be nonsingular. Moreover, suppose that the inequalities in Theorem \ref{DI::coro::ana_RDT2} are feasible. Then the inequalities \eqref{DI::theo::eq::basic_Lyapunov_LMIs} in Theorem \ref{DI::theo::basic_Lyapunov} are feasible as well.
\end{lemma}

\begin{proof}[Sketch of proof]
The proof follows the lines of the one of Lemma \ref{DI::lem::conservativsm_estimate_clock}, but one has to deal with lifted matrices and corresponding inequalities which requires too much space to be presented in detail.
To sketch the ideas, we introduce the matrix
\begin{equation*}
	K := (\bullet)^\top\! \mat{cc|c}{Z & 0 & \\ 0 & -Z & \\ \hline  && I_{k+1} \kron M}
	\mat{cccc}{A_\Psi^{k+1} & A_\Psi^{k}B_\Psi  & \dots & B_\Psi
		\\ I & 0 & \dots  & 0\\ \hline C_\Psi & D_\Psi &   & 0 \\
		\vdots & \vdots & \ddots  &  \\
		C_\Psi A_\Psi^{k} & C_\Psi A_\Psi^{k-1} B_\Psi & \dots  & D_\Psi } 	
	\text{ ~satisfying~ }(\bullet)^\top K \mat{cc}{I & 0 \\ 0 & \mathscr{I}} \cge 0
\end{equation*}
due to \eqref{DI::lem::eq::dynamic_multiplier_RDT2} and with
$\mathscr{I}$ being the same matrix as in the previous proof.
Afterwards, one can incorporate this matrix into the lifted version of the inequality \eqref{DI::theo::eq::maina}. This is rather tedious and, in the case of $T_{\min} = T_{\max} = 1$, the resulting inequality reads as
\begin{equation*}
	\arraycolsep=5pt
	0 \cg (\bullet)^\top \mat{cc|c:c}{\t X & 0 & &\\ 0 & -\t X & &\\ \hline &&  K& \\ \hdashline &&& I_2 \kron P}
	\mat{cc|cc:cc}{
		\bullet & \bullet & \bullet & \bullet & \bullet & \bullet \\
		0 & A^2 & AB_w & B_w & AB & B \\
		I& 0 & 0 & 0 & 0 & 0\\
		0& I & 0 & 0 & 0 & 0\\ \hline
		I & 0 & 0 & 0 & 0 & 0 \\
		0 & C_z & D_{zw} & 0 & D_{zd} & 0 \\
		0 & 0 & I & 0 & 0 & 0 \\
		0 & C_z A & C_z B_w & D_{zw} & C_z B & D_{zd} \\
		0 & 0 & 0 & I & 0 & 0 \\\hdashline
		%
		0 & C & D_{ew} & 0 & D & 0 \\
		0 & 0 & 0 & 0 & I & 0 \\
		0 & CA & CB_w & D_{ew} & CB & D \\
		0 & 0 & 0 & 0 & 0 & I}
\end{equation*}
for $\t X := X - \smat{Z & 0 \\ 0 & 0}\cg 0$ and where the entries ``$\bullet$'' in the matrix on the right are no longer relevant.
Next, one multiplies the right hand side of this inequality from the right with the full column rank matrix
\begin{equation*}
	\mat{c:cc}{I & & \\ \hdashline
		& I & 0  \\
		&\h \Del (I - \h D_{zw}\h \Del)^{-1}\h C_z & \h \Del(I - \h D_{zw}\h \Del)^{-1}\h D_{zd} \\
		&0 & I}
\end{equation*}
and from the left with its transpose; here, the involved (lifted) matrices are given by $\h \Del := \diag(0, \dots, 0, I)$,
\begin{equation*}
	\h C_z := \mat{c}{C_z \\ \vdots \\ C_zA^k}, \quad
	\h D_{zw}:= \mat{ccc}{D_{zw} & & 0 \\ \vdots & \ddots & \\ C_z A^{k-1}B_w & \dots & D_{zw}}
	\teq{ and }
	\h D_{zd}:= \mat{ccc}{D_{zd} & & 0 \\ \vdots & \ddots & \\ C_z A^{k-1}B_d & \dots & D_{zd}}.
\end{equation*}
In the resulting inequalities, the matrices $A_J, B_J, C_J, D_J$ emerge at the desired spots and the sign constraints on $K$ can be exploited. In the
case of $T_{\min} = T_{\max} = 1$, one obtains in this fashion the inequality
\begin{equation*}
	0 \cg (\bullet)^\top \mat{cc|c}{\t X & 0 & \\ 0 & -\t X & \\ \hline && I_2 \kron P}
	\mat{cc|cc}{
		\bullet & \bullet & \bullet & \bullet\\
		0 & A_JA & A_JB & B_J \\
		I& 0 & 0 & 0 \\
		0& I & 0 & 0 \\ \hline
		%
		%
		0 & C & D & 0 \\
		0 & 0 & I & 0 \\
		0 & C_JA & C_JB & D_J \\
		0 & 0 & 0 & I}.
\end{equation*}
The remaining part of the proof relies on identical arguments as those for Lemma \ref{DI::lem::conservativsm_estimate_clock} and is omitted.
\end{proof}

\begin{remark}[Path Based IQCs]
It is not too difficult to derive alternative IQCs building upon the path based approach from \cite{XiaTra19}. However, the resulting performance test involves LMIs that are rather expensive to solve and also seem to be numerically troublesome for some LMI solvers.
\end{remark}

\begin{remark}[Computational Burden]
Let us emphasize that, for standard impulsive systems of the form \eqref{DI::eq::sys_standard} without additional uncertainties, all of the proposed IQC based stability tests are in general computationally more expensive than the available ones in Subsection \ref{DI::sec::recap}. This is mostly due to the additional dynamics of the filter $\Psi$.
However and as illustrated in the next subsection, this changes when dealing with interconnected systems where only few components are affected by impulses.
\end{remark}

\begin{remark}
Note that Theorem \ref{DI::theo::IQC} is somewhat tailored to the impulsive operator $\Del$ in \eqref{DI::eq::sys::del} in the sense that it depends on the underlying sequence of impulse instants. In principle, it would be nicer to find an IQC with terminal cost exactly as defined in \cite{SchVee18,Sch21} for this operator, i.e., with a constant map $Z$ and with the involved sequence defined by $t_k := k-1$.
For example, this can be achieved by exploiting that, for the particular filter \eqref{DI::eq::filter_TF_SS}, we have
\begin{equation*}
	\Psi\left(\mat{c}{z \\ w}\right)(t) = \mat{c}{z(t-\nu) \\ w(t-\nu) \\ \vdots \\ z(t) \\ w(t)} = \diag\left(\mat{c}{I \\ p_{1}I}, \dots,  \mat{c}{I \\ p_{\nu+1} I} \right) \mat{c}{z(t-\nu) \\ \vdots \\ z(t)}
	\teq{ for some }p \in \Ps_{\nu+1}
\end{equation*}
and by choosing the middle matrix $M$ as a suitable static full-block multiplier; here, $\Ps_{\nu+1}$ denotes the set of admissible impulse paths as defined in Theorem \ref{DI::theo::path_Lyapunov}.
However, for this approach one can show that working with $\nu < T_{\max}$ is too restrictive and that, on the other hand, employing $\nu \geq T_{\max}$ often leads to a large computational burden.
\end{remark}

\subsection{Numerical Examples}\label{DI::sec::exa_ana}

\subsubsection{Nominal Stability}

Consider a family of standard discrete-time impulsive systems without inputs and outputs described by
\begin{equation}
\label{DI::eq::exa1}
x(t+1) = \left(\mat{cc}{1 & 0 \\ 0 & 1} + \frac{\beta}{100}\mat{cc}{-1 & - 1 \\ 1 & -1}\right) x(t),\quad
%
%
x(t_k+1) = \mat{cc}{0 & -10 \\ 0.1 & 0}x(t_k)
\end{equation}
for $t \in \N_0 \setminus \{t_1, t_2, \dots\}$, $k \in \N$ and for some parameter $\beta \in [0.1, 5]$. This system is motivated by an example demonstrating the conservatism of employing common quadratic Lyapunov functions for analyzing switched systems in continuous-time as given in \cite{DayMar99}.

We can now analyze these systems based on the specialization to stability analysis of our new tests and of those recapitulated in Section \ref{DI::sec::recap} for several instances of the parameter $\beta$ and of the dwell-time boundaries $T_{\min}, T_{\max}$.
Our numerical findings are illustrated in Fig.~\ref{DI::fig::nom} where we consider an equidistant grid of $50$ points for the parameter interval $[0.1, 5]$. For the IQC based results Theorem \ref{DI::coro::ana_RDT} and \ref{DI::coro::ana_RDT2}, we employ the filter as given in \eqref{DI::eq::filter_TF_SS} with length $\nu = 1$ and, for the path based approach in Theorem \ref{DI::theo::path_Lyapunov}, we consider paths of length $L = 11$. All appearing LMIs are solved via LMIlab \cite{GahNem95} in Matlab on a general purpose desktop computer (Intel Core i7, 4.0 GHz, 8 GB of RAM).

The left hand side of Fig.~\ref{DI::fig::nom} shows that we have obtained identical results for the stability tests based on lifting, the clock and IQCs with lifting as given in Theorems \ref{DI::theo::basic_Lyapunov}, \ref{DI::theo::clock} and \ref{DI::coro::ana_RDT2}, respectively.
We emphasize that only for the first two of these tests this is expected from the underlying theory and that this happens despite the fact that the test in Theorem~\ref{DI::coro::ana_RDT2} is in general more conservative than the one in Theorem~\ref{DI::theo::basic_Lyapunov}.
This figure also demonstrates that all of these three tests can be too conservative, since the path based approach in Theorem~\ref{DI::theo::path_Lyapunov} assures stability for even more parameter combinations, such as for $[T_{\min}, T_{\max}] = [6,9]$ where the other tests fail to guarantee stability. 

Finally, the right hand side of Fig.~\ref{DI::fig::nom} provides a comparison of the IQC based stability tests involving lifting and the clock. It shows that, in contrast to the standard tests in Theorems \ref{DI::theo::basic_Lyapunov} and \ref{DI::theo::clock}, both tests are not equivalent since the clock based approach in Theorem \ref{DI::coro::ana_RDT} yields more conservative results if compared to the one in Theorem \ref{DI::coro::ana_RDT2}.

\begin{figure}
\begin{minipage}{0.49\textwidth}
	\includegraphics[width=\textwidth, trim=15 0 30 20, clip]{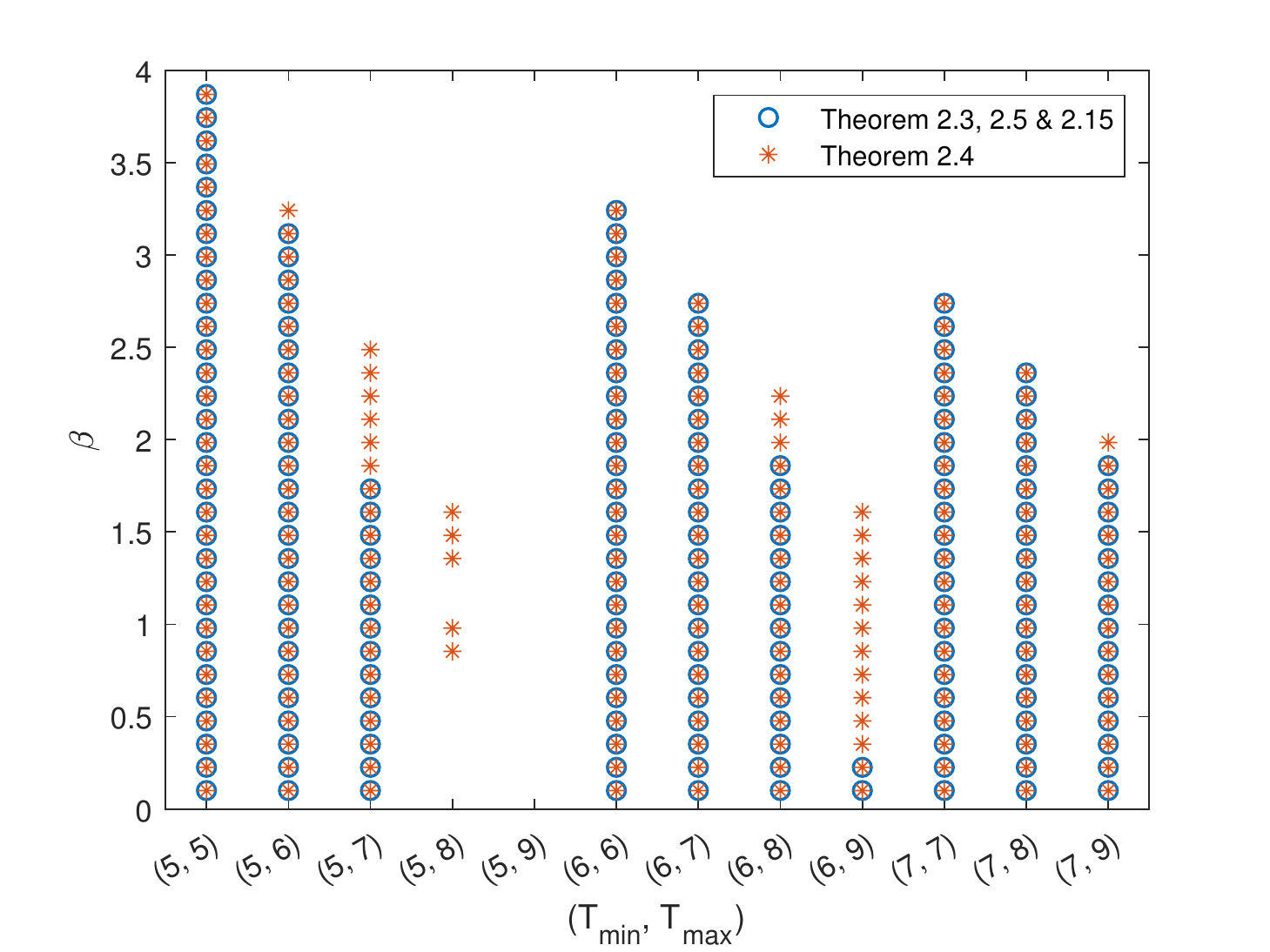}
\end{minipage}
\hfill
\begin{minipage}{0.49\textwidth}
	\includegraphics[width=\textwidth, trim=15 0 30 20, clip]{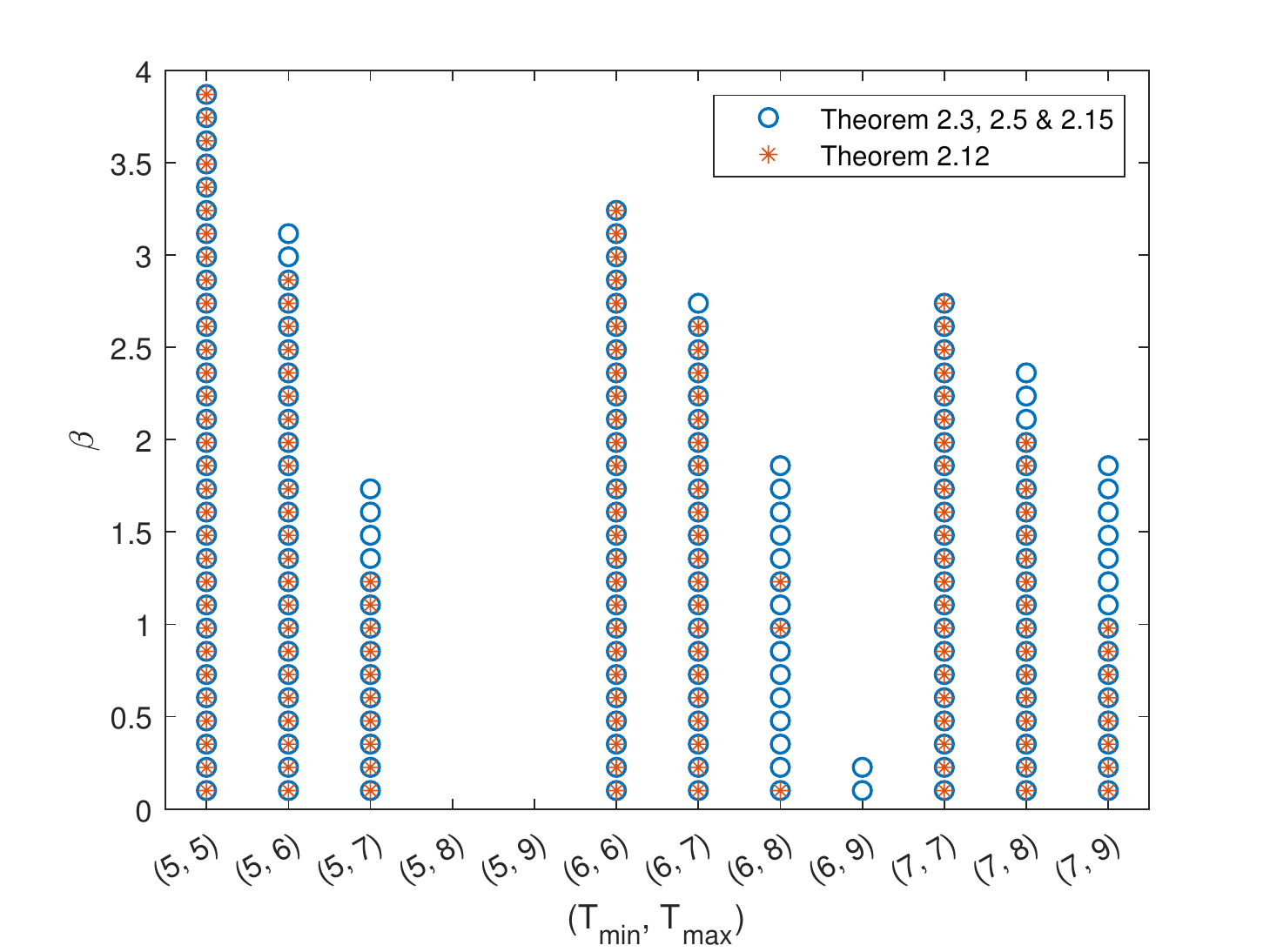}
\end{minipage}
\caption{Parameter combinations of $\beta$, $T_{\min}$ and $T_{\max}$ for which the specializations to stability analysis of the tests in Theorems \ref{DI::theo::basic_Lyapunov}, \ref{DI::theo::path_Lyapunov}, \ref{DI::theo::clock},  \ref{DI::coro::ana_RDT} and \ref{DI::coro::ana_RDT2} assure stability of the impulsive system \eqref{DI::eq::exa1}.}
\label{DI::fig::nom}
\end{figure}

\subsubsection{$H_\infty$ Performance}

As another illustration with some practical flavor, let us consider a simple model for a flexible satellite which is explained in \cite{FraPow10} and, here, modeled as follows with state $\t x = \col(\theta_2, \dot \theta_2, \theta_1, \dot \theta_1)$ and with constants $J_1 = 1$, $J_2 = 0.1$, $k = 0.091$ and $b = 0.04$:
\begin{equation}
\mat{c}{\dot {\t x}(t) \\ \hline v(t)}
= \mat{cccc|c:c}{0 & 1 & 0 & 0 & 0 & 0 \\ -\frac{k}{J_2} & -\frac{b}{J_2} & \frac{k}{J_2} & \frac{b}{J_2} & 1 & 0 \\ 0 & 0 & 0 & 1 & 0 & 0 \\ \frac{k}{J_1} & \frac{b}{J_1} & -\frac{k}{J_1} & -\frac{b}{J_1} & 0 & \frac{1}{J_1} \\ \hline 1 & 0 & 0 & 0 & 0 & 0}
\mat{c}{\t x(t) \\ \hline \t d(t) \\ u(t)}.
\label{DI::eq::sys_ex}
\end{equation}
We discretize this model with a sampling time of $0.01$ seconds and employ the standard $H_\infty$ design procedure to synthesize a (non-impulsive) dynamic output-feedback controller $K$ for the discretized system.
More concretely, we aim to render the closed-loop interconnection stable and to assure that the output $v$ nicely follows a given piecewise constant reference signal $r$ despite the presence of a disturbance $\t d$ and even if the control input $u$ is not too large.
To this end, we consider a standard weighted reference tracking configuration as depicted in Fig.~\ref{NSY::fig::tracking_config} with weights
\begin{equation*}
W_r = 1,\quad
W_d = 0.01, \quad
W_u = 0.1\teq{ and }
%
W_{err}(z)= \frac{0.5(z-0.9913)}{z-1}
\end{equation*}
and where $G$ denotes the discretization of \eqref{DI::eq::sys_ex}.
\begin{figure}
\begin{center}
	\includegraphics[width=0.8\textwidth]{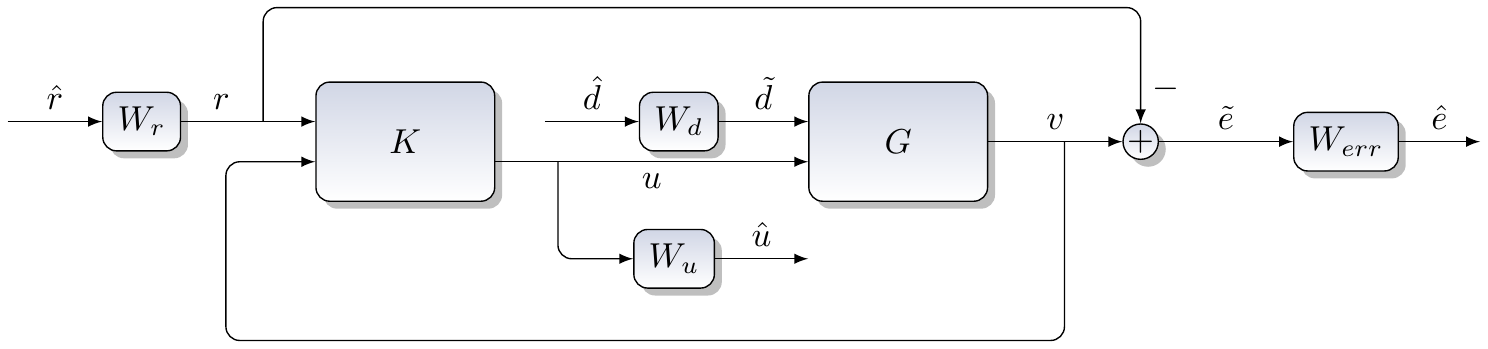}
\end{center}
\vspace{-2ex}
\caption{A standard weighted tracking configuration.}
\label{NSY::fig::tracking_config}
\end{figure}
Disconnecting the controller $K$ from this configuration results in a weighted open-loop system of the form
\begin{equation*}
\mat{c}{x(t+1) \\ \hline e(t) \\ \hdashline y(t)}\mat{c|c:c}{A & B_d & B_u \\ \hline C_e & D_{ed} & D_{eu} \\ \hdashline C_y & D_{yd} & 0}
\mat{c}{x(t) \\ \hline d(t) \\ \hdashline u(t)}
\end{equation*}
with signals
\begin{equation*}
%
e := \mat{c}{\h e \\ \h u},\quad
y := \mat{c}{v \\ r}, \quad
d := \mat{c}{\h r \\ \h d}
\teq{ as well as }u
\end{equation*}
and with easily computed describing matrices.
We then utilize \texttt{hinfsyn} from Matlab which results in a close-to-optimal controller $K$ that achieves a closed-loop energy gain of $0.7986$.

\vspace{0.5ex}

Let us now assume that, due to limited communication, the output $v$ of the system \eqref{DI::eq::sys_ex} can only be measured at times $t_0, t_1, \dots$ with $(t_k)_{k\in \N_0}$ satisfying \eqref{DI::eq::RDT} and that we insist on driving the previously obtained controller $K$ without any modifications by the reference $r$ and the piecewise constant signal
\begin{equation*}
\t v(t) := v(t_k) \teq{ for } t \in [t_k+1, t_{k+1}] \teq{ and }k\in \N_0.
\end{equation*}
In order to analyze the resulting closed-loop, note that the latter signal $\t v$ can be expressed as the output of the following impulsive system $H$ driven by $v$:
\begin{equation}
\label{DI::eq::ex_sys_hybrid}
\mat{c}{x_v(t+1) \\ \t v(t)} = \mat{cc}{I & 0 \\ I & 0}\mat{c}{x_v(t) \\ v(t)},\quad
\mat{c}{x_v(t_k+1) \\ \t v(t_k)} = \mat{cc}{0 & I \\ 0 & I} \mat{c}{x_v(t_k) \\ v(t_k)}
\end{equation}
for $t \in \N_0 \setminus \{t_1, t_2, \dots\}$ and $k \in \N$.
In particular, the resulting closed-loop then is an (interconnected) linear impulsive system that is depicted in Fig.~\ref{DI::fig::tracking_config_hybrid} and whose performance can be analyzed based on any of the available tests in Section \ref{DI::sec::recap} or on the new ones in Theorem \ref{DI::coro::ana_RDT} and \ref{DI::coro::ana_RDT2}.
Since the controller $K$ was designed with the aim to achieve a small energy gain for the interconnection in Fig.~\ref{NSY::fig::tracking_config}, we consider quadratic performance with the corresponding performance index $P_\ga := \smat{I & 0 \\ 0 & -\ga^2 I}$ and minimize over $\ga$. This yields (optimal) upper bounds on the actual energy gain which depend on the employed test and on the dwell-time range $[T_{\min}, T_{\max}]$.

\begin{figure}
\begin{center}
	\includegraphics[width=0.8\textwidth]{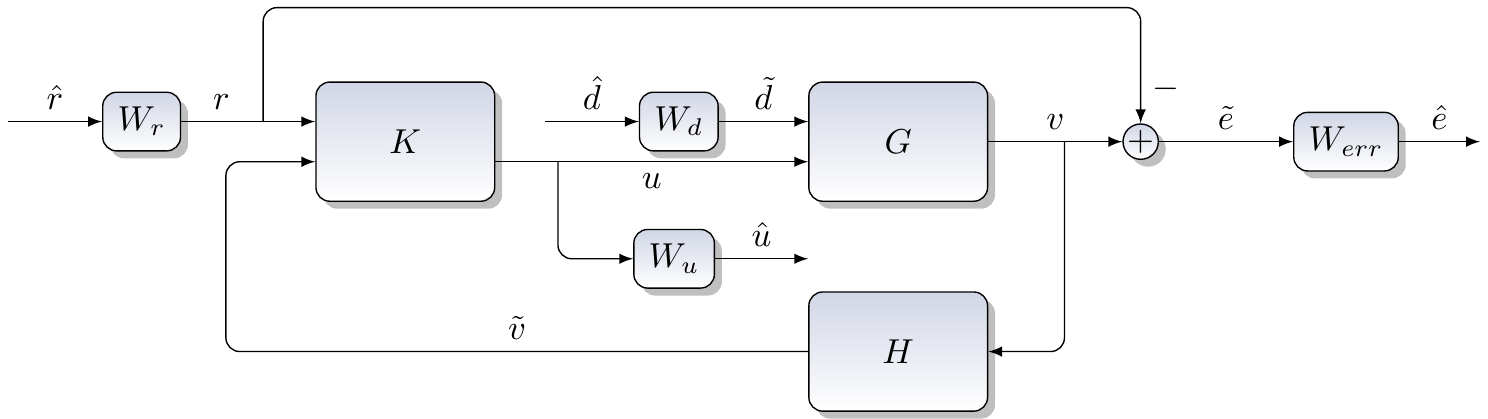}
\end{center}
\vspace{-2ex}
\caption{A weighted tracking configuration involving the impulsive component \eqref{DI::eq::ex_sys_hybrid}.}
\label{DI::fig::tracking_config_hybrid}
\end{figure}

\begin{figure}
\begin{minipage}{0.49\textwidth}
	\begin{center}
		\includegraphics[trim=20 0 30 20, clip, width=\textwidth]{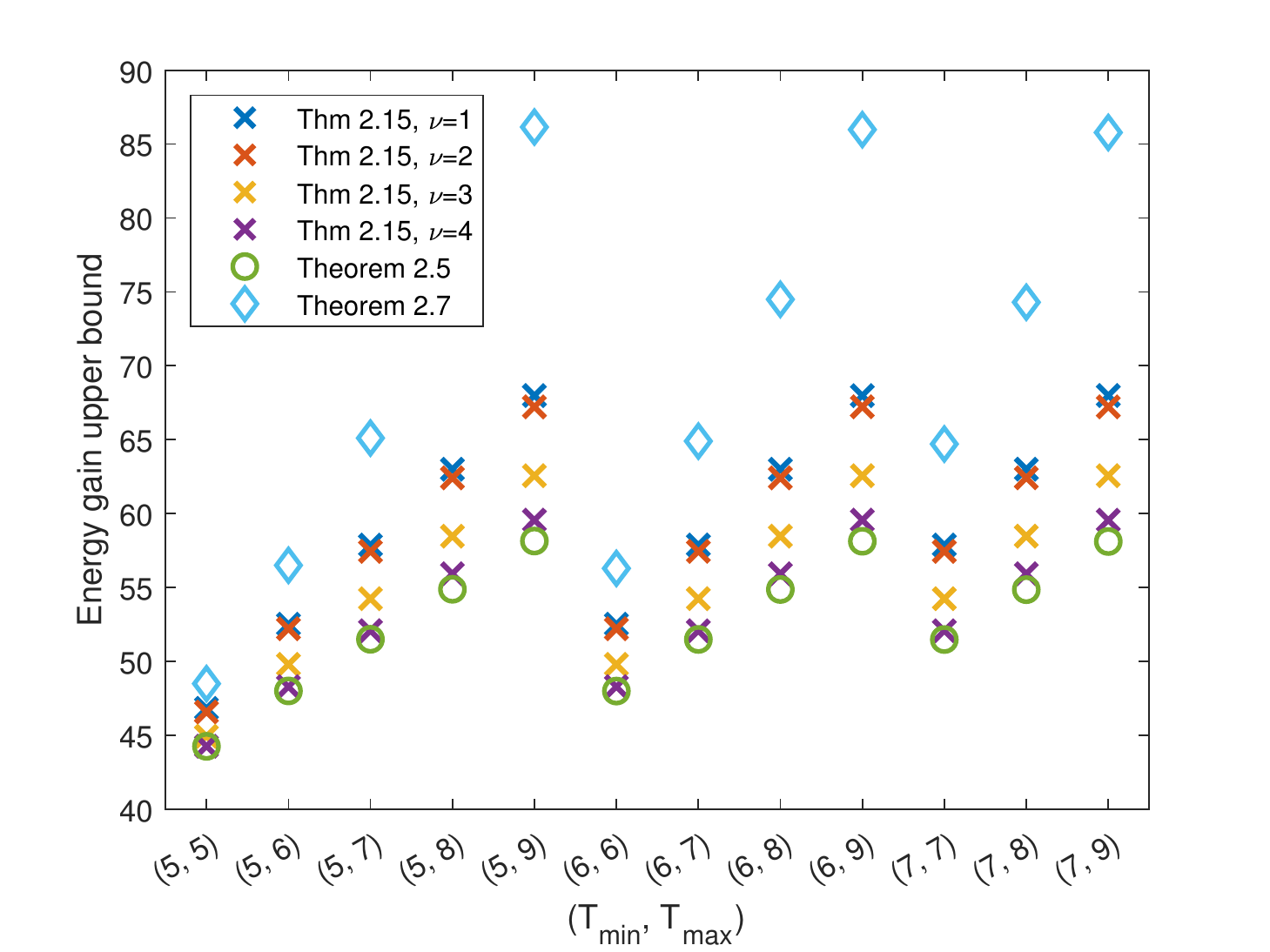}
	\end{center}
\end{minipage}
\begin{minipage}{0.49\textwidth}
	\begin{center}
		\includegraphics[trim=20 0 30 20, clip, width=\textwidth]{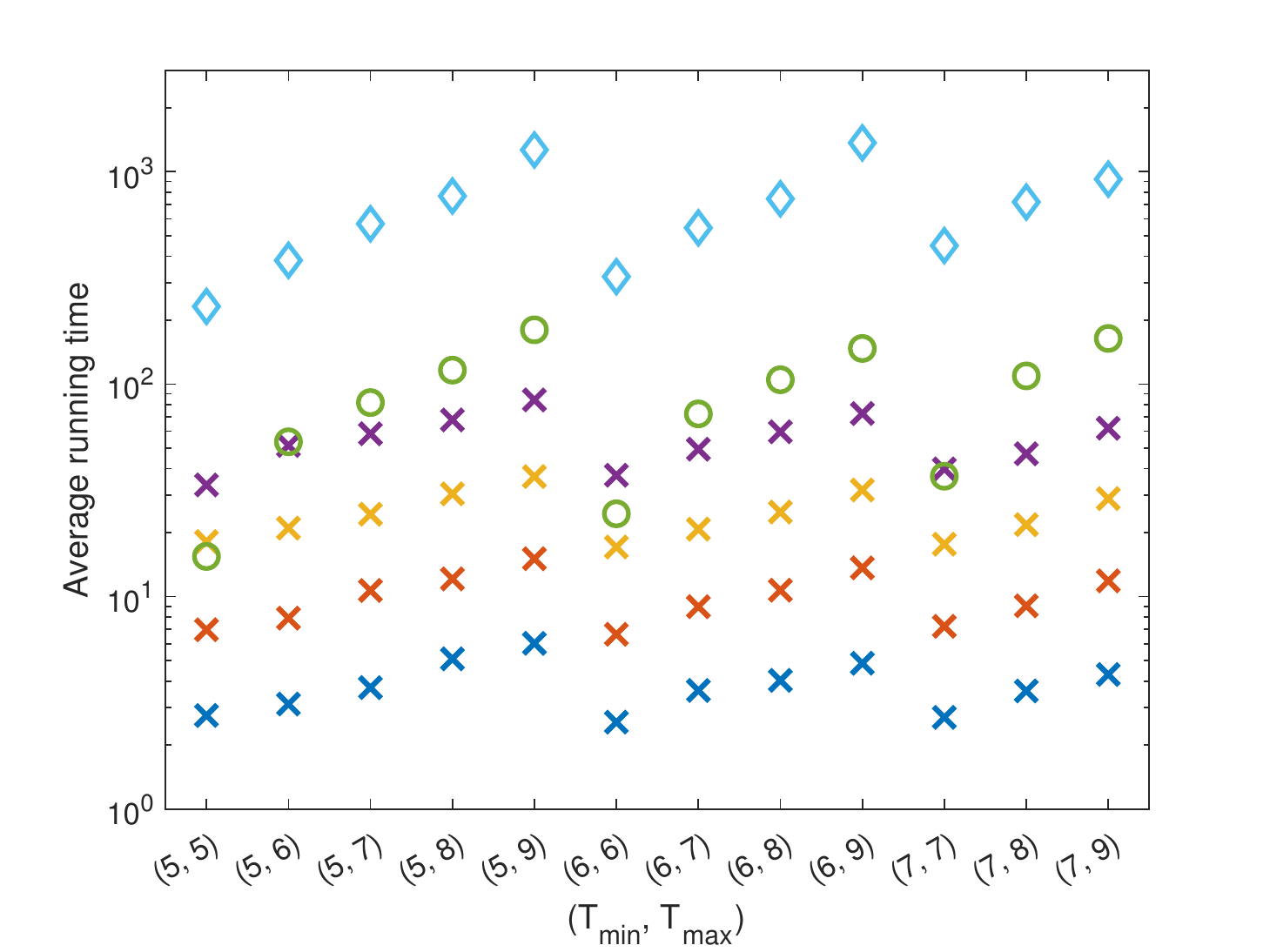}
	\end{center}
\end{minipage}
\vspace{-2ex}
\caption{Left: Numerically determined upper bounds on the energy gain of the interconnection in Fig.~\ref{DI::fig::tracking_config_hybrid} as obtained from Theorem \ref{DI::theo::clock}, \ref{DI::theo::clock_sv} and Theorem~\ref{DI::coro::ana_RDT2} with several filter lengths $\nu$. Right: Corresponding average running times in seconds within twenty runs.}
\label{DI::fig::energy_gain}
\end{figure}

On the left in Fig.~\ref{DI::fig::energy_gain} are depicted upper bounds on the energy gain for several pairs of $T_{\min}$ and $T_{\max}$ as numerically determined by the clock based performance test in Theorem \ref{DI::theo::clock}, the one involving common slack variables in Theorem \ref{DI::theo::clock_sv} as well as by the IQC and lifting based test in Theorem \ref{DI::coro::ana_RDT2}. In the latter test, we employ the filter \eqref{DI::eq::filter_TF_SS} with length $\nu \in \{1, \dots, 4\}$.
On the right in this figure are illustrated the corresponding average running times in seconds within twenty runs.

First of all, we see observe that the test in Theorem \ref{DI::coro::ana_RDT2} provides more conservative upper bounds than the one in Theorem \ref{DI::theo::clock} which is expected from Lemma \ref{DI::lem::conservativsm_estimate_lifting} and from the third statement of Theorem \ref{DI::theo::clock}.
However, if we increase the length of the filter Theorem \ref{DI::coro::ana_RDT2}, then this conservatism is reduced, and at $\nu = 4$ both tests provide nearly the same upper bounds. Moreover, we observe that the IQC based test is faster in almost all instances with $\nu \in \{1, \dots, 4\}$; naturally, this changes if we would further increase the filter length $\nu$.
Finally, the clock based test with common slack variables in Theorem \ref{DI::theo::clock_sv}, which forms the basis of some robust design approaches, is outperformed by the other tests in terms of guaranteed upper bounds and also in terms of the involved numerical burden.
This is one of the main reasons for performing robust design based on the IQC approach instead.
As a final side note, since the guaranteed closed-loop performance became rather poor if compared to the originally achieved one without limited communication, one should take the effort to design a new controller $K$ that takes the impulsive nature of the interconnection depicted in Fig.~\ref{DI::fig::tracking_config_hybrid} into account. This can for example be done as demonstrated in \cite{Hol22}.

\section{Estimator Design}\label{DI::sec::es}

For numerous classes of uncertainties, it has been shown in previous work that the problem of synthesizing robust estimators for uncertain systems can be convexified in case that these systems admit a linear fractional representation and if the properties of the uncertainties are captured by IQCs \cite{SchKoe08, ScoFro06, Vee15, GerDeo01, Ger99, VenSei16}. This section is motivated by the latter fact and serves to demonstrate that our IQC-based analysis results permit us to design non-impulsive estimators for impulsive systems by convex techniques.

\subsection{Problem Description}\label{DI::sec::prob_set_es}

For real matrices of appropriate dimensions, some initial condition $x(0)\in \R^n$, some generalized disturbance $d\in \ell_2^{n_d}$ and the impulsive operator $\Del$ defined in \eqref{DI::eq::sys::del} through a sequence of impulse instants $(t_k)_{k\in \N_0}$, we consider the impulsive system
\begin{equation}
	\mat{c}{x(t+1) \\ \hline  z(t) \\  v(t) \\  y(t)}
	= \mat{c|cc}{A & B_w & B_d  \\ \hline
		C_z & D_{zw} & D_{zd} \\
		C_v & D_{vw} & D_{vd}  \\
		C_y & D_{yw} & D_{yd}}
	\mat{c}{x(t) \\ \hline w(t) \\  d(t) }, \qquad
	w(t) = \Del(z)(t)
	\label{DI::eq::es_sys}
\end{equation}
for $t \in \N_0$. Here, the new signals $v$ and $y$ denote some output that we aim to estimate and the measured output which drives the estimator, respectively.

The precise goal in this section is the design of a \emph{non-impulsive} estimator $\Ec$ described by the LTI system
\begin{equation}
	\mat{c}{x_e(t+1) \\  u(t)} =
	\mat{cc}{A^e & B^e \\ C^e & D^e} \mat{c}{x_e(t)  \\ y(t)}
	\label{DI::eq::es_con}
\end{equation}
for $t \in \N_0$ by means of convex optimization, which takes the measured signal $y$ as its input and generates an optimal approximation $u$ of the signal $v$; in the sequel, we measure the approximation quality in terms of the energy gain from the disturbance input $d \in \ell_2^{n_d}$ to the estimation error
\begin{equation}
	\label{DI::eq::es_error}
	e := v - u.
\end{equation}
Note that this problem is much more challenging than designing an impulsive estimator that is aware of the time instances at which impulses occur as considered for example in \cite{MedLaw09, BerSan18, ConPer17}, since much less information about the system \eqref{DI::eq::es_sys} is available. Recall that knowledge of the sequence $(t_k)_{k\in \N_0}$ can be essential for the purpose of estimation because the behavior of some impulsive systems can change dramatically for different dwell-times $t_{k+1}-t_k+1$.

In order to design such a non-impulsive estimator, we consider the corresponding closed-loop interconnection of the estimator \eqref{DI::eq::es_con} with the system \eqref{DI::eq::es_sys} and with the performance output \eqref{DI::eq::es_error}. This interconnection can be expressed as
\begin{equation}
	\mat{c}{x_{cl}(t+1) \\ \hline z(t) \\ e(t)}
	=\mat{c|cc}{\Ac & \Bc_w & \Bc_d \\ \hline
		\Cc_z & \Dc_{zw} & \Dc_{zd} \\
		\Cc_e & \Dc_{ew} & \Dc_{ed}}
	\mat{c}{x_{cl}(t) \\ \hline w(t) \\ d(t)}, \qquad
	w(t) = \Del(z)(t)
	\label{DI::eq::es_cl}
\end{equation}
for $t\in \N_0$ and with the stacked state $x_{cl} = \smat{x\\ x_e}$ as well as the calligraphic closed-loop matrices given by
\begin{multline}
	\mat{cc|cc}{A  & 0 & B_w & B_d\\
		B^e C_y & A^e & B^e D_{yw}  & B^e D_{yd}\\ \hline
		C_z & 0 & D_{zw} & D_{zd}\\
		C_v -  D^e C_y & C^e & D_{vw} -  D^e D_{yw} & D_{vd} - D^e D_{yd}} \\
	= \mat{cc|cc}{A & 0 & B_w & B_d\\
		0 & 0 & 0 & 0 \\ \hline
		C_z & 0 & D_{zw} & D_{zd}\\
		C_v & 0 & D_{vw} & D_{vd}} +
	\mat{cc}{0 & 0 \\ I & 0 \\ \hline 0 & 0 \\ 0 & -I}
	\mat{cc}{A^e & B^e \\ C^e & D^e}
	\mat{cc|cc}{0 & I & 0 & 0 \\ C_y & 0 & D_{yw} & D_{yd}}.
	\label{DI::eq::matrices_es_cl}
\end{multline}
This interconnection is also depicted in Fig.~\ref{DI::fig::es} where $G_{oi} := (A, B_i, C_o, D_{oi})$ for $o \in \{z, v, y\}$ and $i \in \{w, d\}$ denote the transfer matrices corresponding to \eqref{DI::eq::es_sys}.

\begin{figure}
	\begin{center}
		\includegraphics[]{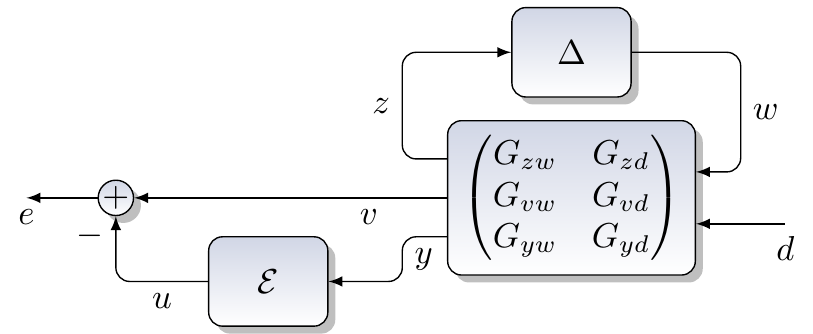}
	\end{center}
	\caption{Block diagram corresponding to the closed-loop interconnection \eqref{DI::eq::es_cl}.}
	\label{DI::fig::es}
\end{figure}

Since the description \eqref{DI::eq::es_cl} is exactly of the form \eqref{DI::eq::sys}, we can apply any of the performance tests in Corollary~\ref{DI::coro::ana_ADT} or in Theorems \ref{DI::theo::basic_Lyapunov}, \ref{DI::theo::path_Lyapunov} \ref{DI::theo::clock}, \ref{DI::theo::clock_sv}, \ref{DI::coro::ana_RDT}, \ref{DI::coro::ana_RDT2} for its analysis depending on the available information on the sequence of impulse instants. For brevity, we assume that this sequence is known to satisfy \eqref{DI::eq::RDT} and only formulate the extension of Theorem \ref{DI::coro::ana_RDT2}. Since we are interested in the energy gain, note that we consider the corresponding quadratic performance index $P_\ga := \smat{I & 0 \\ 0 & -\ga^2 I}$ instead of a general one.

\begin{corollary}
	\label{DI::lem::es_ana}
	Let $\Psi$ be some filter as in \eqref{DI::eq::filter}. Then the closed-loop interconnection \eqref{DI::eq::es_cl} achieves an energy gain smaller than $\ga$ for all sequences $(t_k)_{k\in \N_0}$ with \eqref{DI::eq::RDT}
	if there exist matrices $\Xc \in \S^{n_\xi + n + n_e}$, $Z \in \S^{n_\xi}$ and $M\in \S^{m}$ satisfying \eqref{DI::lem::eq::dynamic_multiplier_RDT2},
	\begin{subequations}
		\label{DI::lem::eq::es_main}
		\begin{equation}
			\dlabel{DI::lem::eq::es_maina}{DI::lem::eq::es_mainb}
			\Xc - \mat{cc}{Z & 0 \\ 0 & 0}\cg 0
			\teq{ and }
			(\bullet)^\top \mat{cc|c:c}{\Xc & 0 & &\\ 0 & -\Xc & &\\ \hline &&  M& \\ \hdashline &&& P_\ga} \mat{cc|cc}{A_\Psi & B_\Psi \smat{\Cc_z \\ 0} & B_\Psi \smat{\Dc_{zw} \\ I_{n_w}} & B_\Psi\smat{\Dc_{zd} \\ 0}
				\\ 0 & \Ac & \Bc_w & \Bc_d
				\\ I& 0 & 0 & 0\\ 0& I & 0 & 0\\ \hline
				C_\Psi & D_\Psi \smat{\Cc_z \\ 0} & D_\Psi \smat{\Dc_{zw} \\ I_{n_w}} & D_\Psi \smat{\Dc_{zd} \\ 0} \\ \hdashline
				0 & \Cc_e & \Dc_{ew} & \Dc_{ed} \\ 0 & 0& 0 & I} \cl 0.
		\end{equation}
	\end{subequations}
\end{corollary}

Note that this analysis result cannot efficiently be used for designing an estimator \eqref{DI::eq::es_con} since simultaneously searching for its describing matrices and the matrices $X$, $Z$, $M$ satisfying the above inequalities is a difficult nonconvex problem. This is resolved in the following subsection.

\subsection{Main Synthesis Result}\label{DI::sec::main_es}

Our main result in this section is a novel convex characterization of LTI estimators \eqref{DI::eq::es_con} satisfying the closed-loop analysis inequalities in Corollary \ref{DI::lem::es_ana}.
Since the latter criteria are based on IQCs, our synthesis result is related, e.g., to the ones in \cite{SchKoe08, Vee15, ScoFro06} that employ IQCs as well, but for non-impulsive uncertain systems.
For example in \cite{SchKoe08}, convexification rests upon a factorization of the involved IQC multiplier $\Pi$ and the convexifying parameter transformation \cite{MasOha98, SchGah97} which is well-known in the LMI literature.
Instead, we rely on the elimination lemma \cite[Theorem 2]{Hel99} which is restated here for the reader's convenience.

\begin{lemma}[Elimination Lemma]
	\label{DI::lem::elimination}
	Let $P \in \S^l$ and let $W\in \R^{l\times k}$, $U\in \R^{l\times m}$, $V \in \R^{n\times k}$. Then there exists a matrix $K \in \R^{m \times n}$ such that
	\begin{equation*}
		\label{LMI_with_Z}
		(\bullet)^\top P (W + U KV) \cl 0
	\end{equation*}
	holds if and only if
	\begin{equation*}
		\label{LMI_without_Z}
		(\bullet)^\ast P WV_\perp \cl 0
		\teq{ and }
		(\bullet)^\top P \mat{cc}{W & U}
		\text{ has at least $k$ negative eigenvalues.}
	\end{equation*}
\end{lemma}

In \cite{Vee15}, the more common variant of the elimination lemma as given in \cite[Theorem 3]{Hel99} is employed together with a factorization of the IQC multiplier. This variant requires slightly more structure of the involved outer factors, but fits perfectly well, e.g., to standard $H_\infty$-control with a nice control theoretic interpretation.

Let us emphasize that, due to the application of the above version of the elimination lemma, the proof of our main result does neither require a factorization of the involved IQC multiplier nor an application of the dualization lemma \cite{SchWei00}. This is in stark contrast to \cite{SchKoe08, Vee15} and leads to a much cleaner line of reasoning.

\begin{theorem}
	\label{DI::theo::es_main}
	Let $\t V$ be a basis matrix of $\ker(C_y, D_{yw}, D_{yd})$ and define $V := \diag(I_{n_\xi}, \t V)$.
	Then there exists an estimator \eqref{DI::eq::es_con} for the system \eqref{DI::eq::es_sys} such that the analysis inequalities \eqref{DI::lem::eq::es_main} are satisfied for their closed-loop interconnection \eqref{DI::eq::es_cl} if and only if there exist $X\in \S^{n_\xi+n}$, $Y\in \S^{n_\xi+n}$, $Z \in \S^{n_\xi}$ and $M \in \S^m$ satisfying \eqref{DI::lem::eq::dynamic_multiplier_RDT2},
	\begin{subequations}
		\begin{equation}
			\dlabel{DI::theo::eq::es_maina}{DI::theo::eq::es_mainb}
			\mat{cc}{\t X & \t Y \\ \t Y & \t Y} \cg 0, \quad
			(\bullet)^\top \mat{cc|c:c}{X & 0 & &\\ 0 & -X & &\\ \hline &&  M& \\ \hdashline &&& P_\ga} \mat{cc|cc}{A_\Psi & B_\Psi \smat{C_z \\ 0} & B_\Psi \smat{D_{zw} \\ I_{n_w}} & B_\Psi\smat{D_{zd} \\ 0}
				\\ 0 & A & B_w & B_d
				\\ I& 0 & 0 & 0\\ 0& I & 0 & 0\\ \hline
				C_\Psi & D_\Psi \smat{C_z \\ 0} & D_\Psi \smat{D_{zw} \\ I_{n_w}} & D_\Psi \smat{D_{zd} \\ 0} \\ \hdashline
				0 & C_v & D_{vw} & D_{vd} \\ 0 & 0& 0 & I}V \cl 0
		\end{equation}
		and
		\begin{equation}
			\label{DI::theo::eq::es_mainc}
			(\bullet)^\top \mat{cc|c:c}{Y & 0 & &\\ 0 & -Y & &\\ \hline &&  M& \\ \hdashline &&& -\ga^2I} \mat{cc|cc}{A_\Psi & B_\Psi \smat{C_z \\ 0} & B_\Psi \smat{D_{zw} \\ I_{n_w}} & B_\Psi\smat{D_{zd} \\ 0}
				\\ 0 & A & B_w & B_d
				\\ I& 0 & 0 & 0\\ 0& I & 0& 0\\ \hline
				C_\Psi & D_\Psi \smat{C_z \\ 0} & D_\Psi \smat{D_{zw} \\ I_{n_w}} & D_\Psi \smat{D_{zd} \\ 0} \\ \hdashline
				0 & 0& 0 & I} \cl 0,
		\end{equation}
	\end{subequations}
	where $\t X := X - \diag(Z, 0)$ and $\t Y := Y - \diag(Z, 0)$.
\end{theorem}

\begin{proof}
	{\itshape Necessity:} Without loss of generality, we can assume that $n_e \geq n+n_\xi$ holds; otherwise, one can replace the describing matrices of the estimator
	\begin{equation*}
		\mat{cc}{A^e & B^e \\ C^e & \bullet}
		\teq{ by }\mat{cc|c}{A^e & 0 & B^e \\ 0 & 0 & 0 \\ \hline C^e &  0 & \bullet}
	\end{equation*}
	and finds that the closed-loop analysis inequalities \eqref{DI::lem::eq::es_main} still holds for $\Xc$ replaced by $\diag(\Xc, \alpha I)$ with $\alpha > 0$ sufficiently large. Moreover, we can slightly perturb $\Xc = \smat{X & X_{12} \\ X_{21} & X_{22}}$ with $X \in \S^{n_\xi+n}$ such that these inequalities remain valid and such that $X_{21}$ has full column rank.

	\begin{figure}
		\begin{minipage}{0.49\textwidth}
			\begin{center}
				\includegraphics[width=\textwidth]{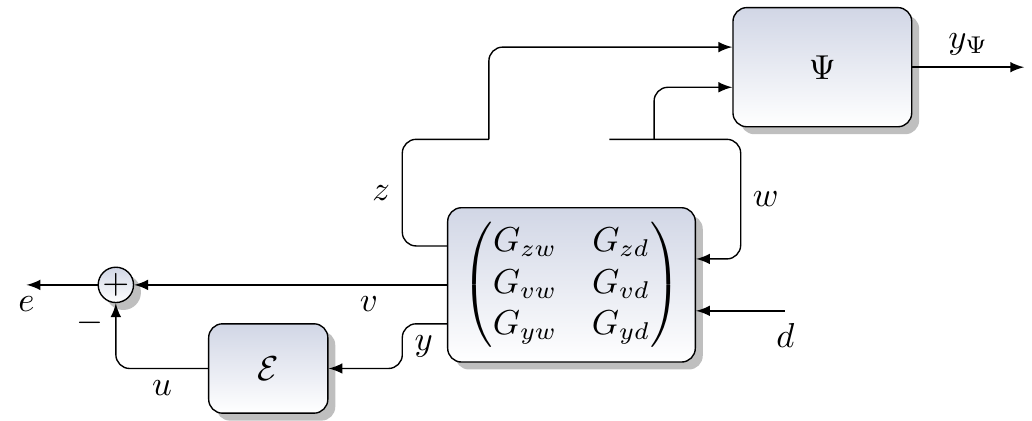}
			\end{center}
		\end{minipage}
		\begin{minipage}{0.49\textwidth}
			\begin{center}
				\includegraphics[width=\textwidth]{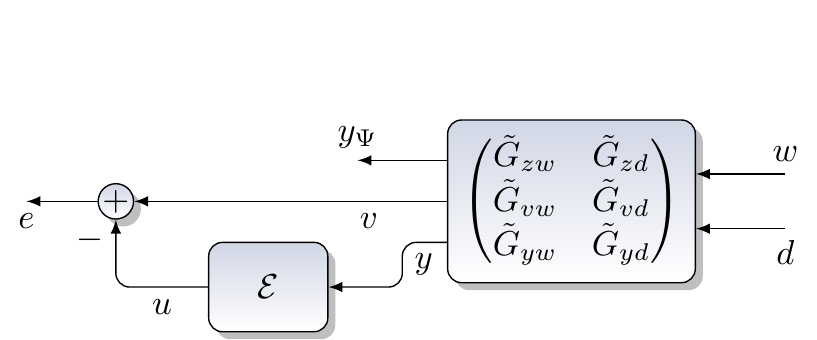}
			\end{center}
		\end{minipage}
		\caption{Two equivalent block diagrams corresponding to the augmentation of the known linear part of the closed-loop interconnection \eqref{DI::eq::es_cl} by some filter $\Psi$.}
		\label{DI::fig::es_augm}
	\end{figure}
	
	Next,
	recall that
	\begin{equation*}
		\mat{cc|cc}{A_\Psi & B_\Psi \smat{\Cc_z \\ 0} & B_\Psi \smat{\Dc_{zw} \\ I_{n_w}} & B_\Psi\smat{\Dc_{zd} \\ 0}
			\\ 0 & \Ac & \Bc_w & \Bc_d \\
			\hline
			C_\Psi &  D_\Psi \smat{\Cc_z \\ 0} &  D_\Psi \smat{\Dc_{zw} \\ I_{n_w}} &  D_\Psi \smat{\Dc_{zd} \\ 0} \\ \hdashline
			0 & \Cc_e & \Dc_{ew} & \Dc_{ed}}
		= \mat{ccc|cc}{A_\Psi & B_{\Psi}\smat{C_z \\ 0} & 0 & B_{\Psi}\smat{D_{zw} \\ I} & B_{\Psi}\smat{D_{zd} \\ 0} \\
			0 & A & 0 & B_w & B_d \\
			0 & B^e C_y & A^e & B^e D_{yw} & B^e D_{yd} \\ \hline
			C_\Psi &  D_{\Psi}\smat{C_z \\ 0} & 0 &  D_{\Psi}\smat{D_{zw} \\ I} &  D_{\Psi}\smat{D_{zd} \\ 0} \\ \hdashline
			0 & C_v - D^e C_y & C^e & D_{vw} - D^e D_{yw} & D_{vd} - D^eD_{yd}}
	\end{equation*}
	denote the describing matrices of the closed-loop's linear part augmented by the filter $\Psi$, which corresponds to the block diagram on the left in Fig.~\ref{DI::fig::es_augm}. This figure also illustrates that this augmented closed-loop still corresponds to an estimation problem; the related matrices are actually given by \eqref{DI::eq::matrices_es_cl} if we replace the describing matrices of the original system \eqref{DI::eq::es_sys} by the following ones of the filtered  system:
	\begin{equation*}
		\mat{c|cc}{\t A & \t B_w & \t B_d \\ \hline \t C_z & \t D_{zw} & \t D_{zd} \\ \t C_v & \t D_{vw} & \t D_{vd} \\ \t C_y & \t D_{yw} & \t D_{yd}}
		= \mat{cc|cc}{A_\Psi & B_\Psi \smat{C_z \\ 0} & B_\Psi \smat{D_{zw} \\ I} & B_\Psi \smat{D_{zd} \\ 0} \\ 0 & A & B_w & B_d \\ \hline  C_\Psi &  D_\Psi \smat{C_z \\ 0} &  D_\Psi \smat{D_{zw} \\ I} &  D_\Psi \smat{D_{zd} \\ 0} \\
			0 & C_v & D_{vw} & D_{vd} \\
			0 & C_y & D_{yw} & D_{yd}}.
	\end{equation*}
	Hence, we can express the outer factor of the closed-loop analysis inequality \eqref{DI::lem::eq::es_mainb} as
	\begin{equation*}
		\arraycolsep=3pt
		\mat{cc|cc}{A_\Psi & B_\Psi \smat{\Cc_z \\ 0} & B_\Psi \smat{\Dc_{zw} \\ I_{n_w}} & B_\Psi\smat{\Dc_{zd} \\ 0}
			\\ 0 & \Ac & \Bc_w & \Bc_d
			\\ I& 0 & 0 & 0\\ 0& I & 0&0\\ \hline
			C_\Psi &  D_\Psi \smat{\Cc_z \\ 0} &  D_\Psi \smat{\Dc_{zw} \\ I_{n_w}} &  D_\Psi \smat{\Dc_{zd} \\ 0} \\ \hdashline
			0 & \Cc_e & \Dc_{ew} & \Dc_{ed} \\ 0 & 0& 0 & I}
		=\underbrace{\mat{cc|cc}{\t A & 0 & \t B_w & \t B_d\\
				0 & 0 & 0 & 0 \\ I & 0 & 0 & 0 \\ 0 & I & 0 & 0 \\ \hline
				\t C_z & 0 & \t D_{zw} & \t D_{zd}\\
				\t C_v & 0 & \t D_{vw} & \t D_{vd} \\ 0 & 0 & 0 & I}}_{=: W} +
		\underbrace{\mat{cc}{0 & 0 \\ I & 0 \\ 0 & 0 \\ 0 & 0 \\ \hline 0 & 0 \\ 0 & -I\\ 0 & 0}}_{=:U}
		\underbrace{\mat{cc}{A^e & B^e \\ C^e & D^e}}_{=: K}
		\underbrace{\mat{cc|cc}{0 & I & 0 & 0 \\ \t C_y & 0 & \t D_{yw} & \t D_{yd}}}_{=:V}.
	\end{equation*}
	Next, let us partition the annihilator $\t V$ accordingly to $(C_y, D_{yw}, D_{yd})$ and note that
	\begin{equation*}
		\h V:= \mat{cc}{I_{n_\xi} & 0\\ 0_{n\times n_\xi} & \t V_1 \\ 0_{n_v \times n_\xi} & 0 \\ 0_{n_w \times n_\xi} & \t V_2 \\ 0_{n_d \times n_\xi} & \t V_3} \teq{ is a basis matrix of }\ker(V).
	\end{equation*}
	Finally, let us define $P := \diag(\Xc, -\Xc, M, P_\ga)$, the middle matrix of \eqref{DI::lem::eq::es_mainb}, and suppose that $W$ has $k$ columns.
	This permits us to apply the elimination lemma \ref{DI::lem::elimination} and allows us to infer from \eqref{DI::lem::eq::es_mainb} that
	\begin{equation*}
		(\bullet)^\top PW \h V \cl 0
		\teq{ and that } (\bullet)^\top P \mat{cc}{W & U} \text{ has at least $k$ negative eigenvalues.}
	\end{equation*}
	By recalling the partition of $\Xc = \smat{X & X_{12} \\ X_{21} & X_{22}}$ and a direct computation, one finds that $(\bullet)^\top P W \h V \cl 0$ is exactly the inequality \eqref{DI::theo::eq::es_mainb}.
	
	Next, one observes that $U^\top P U = \smat{X_{22} & 0 \\ 0 & I} \cg 0$ holds by \eqref{DI::lem::eq::es_maina}. Hence, the matrix $(\bullet)^\top P(W\ \ U)$ has the correct number of negative eigenvalues if and only if its Schur complement w.r.t. $U^\top PU$ satisfies
	\begin{equation*}
		S := W^\top P W - W^\top PU(U^\top P U)^{-1} U^\top P W = W^\top [P - PU(U^\top PU)^{-1}U^\top P] W \cl 0.
	\end{equation*}
	This actually reads as
	\begin{equation*}
		W^\top \diag\left(\mat{cc}{X_{12} - X_{12}X_{22}^{-1}X_{21} & 0 \\ 0 & 0 }, \mat{cc}{-X_{11} & - X_{12} \\ -X_{21} & - X_{22}}, M, 0, -\ga^2 I\right) W\cl 0.
	\end{equation*}
	After a congruence transformation with $\diag\left(\smat{I & 0 \\ -X_{22}^{-1}X_{21} & I}, I_{n_w}, I_{n_d} \right)$ and due to the structure of the matrix $W$, we obtain 
	\begin{equation*}
		W^\top \diag\left(\mat{cc}{X_{12} - X_{12}X_{22}^{-1}X_{21} & 0 \\ 0 & 0 }, \mat{cc}{-(X_{11} -X_{12}X_{22}^{-1}X_{21}) & 0 \\ 0 & - X_{22}}, M, 0, -\ga^2 I\right) W \cl 0
	\end{equation*}
	Due to this identity, the particular structure of the matrix $W$ and by $-X_{22} \cl 0$, one observes that $S \cl 0$ holds if and only if  \eqref{DI::theo::eq::es_mainc} holds for $Y := X - X_{12}X_{22}^{-1}X_{21}$.

	Since \eqref{DI::lem::eq::dynamic_multiplier_RDT2} holds by assumption, it remains to show that the coupling condition \eqref{DI::theo::eq::es_maina} holds for our choice of $X$ and $Y$. Indeed, $\t Y = Y - \smat{Z & 0 \\ 0 & 0}\cg 0$ follows from applying the Schur complement to
	\eqref{DI::lem::eq::es_maina}. By taking another Schur complement, the coupling condition \eqref{DI::theo::eq::es_maina} is then equivalent to
	\begin{equation*}
		0 \cl \t X - \t Y\t Y^{-1}\t Y = \t X - \t Y = X_{21}^\top X_{22}^{-1} X_{21}.
	\end{equation*}
	The latter is inequality is true by $X_{22} \cg 0$ and since $X_{21}$ has full column rank.
	
	\vspace{1ex}
	
	{\itshape Sufficiency:} This follows essentially from reversing the arguments and leads in a constructive fashion to an estimator having at most McMillian degree $n + n_\xi$. It remains to recall that, for given matrices $X$ and $Y$ with \eqref{DI::theo::eq::es_maina}, the matrix
	\begin{equation*}
		\Xc := \mat{cc}{X& X_{12} \\ X_{21} & X_{22}}:=\mat{cc}{X & Y- X \\ Y -X & X-Y}
		\teq{ satisfies } X - X_{12}X_{22}^{-1}X_{21} =Y
		\teq{ and }\eqref{DI::lem::eq::es_maina}
	\end{equation*}
	Indeed, from \eqref{DI::theo::eq::es_maina} we infer $0 \cl \t X- \t Y = X - Y$ by Schur and, by taking another Schur complement, \eqref{DI::lem::eq::es_maina} is then equivalent to
	\begin{equation*}
		0 \cl X - \smat{Z & 0 \\ 0 & 0} + (Y - X) = Y - \smat{Z & 0 \\ 0 & 0} = \t Y.
	\end{equation*}
	The latter inequality is true by \eqref{DI::theo::eq::es_maina}.
\end{proof}

\subsection{An Alternative Result}\label{DI::sec::es_alternative}

Before providing a numerical example, let us provide an alternative design result that is a discrete-time version of \cite[Theorem 3.11]{Hol22} for the purpose of comparing both design results. This result stems from the clock based quadratic performance test in Theorem \ref{DI::theo::clock} and relies on the introduction of (partially) common slack variables similarly is in Theorem \ref{DI::theo::clock_sv}. A conceptually related estimator design approach involving such slack variables can be found, e.g., in \cite{GerDeo02}.
This alternative design result deals with an open-loop impulsive system with description
\begin{equation}
	\label{DI::eq::es_sys_alternative}
	\mat{c}{x(t+1) \\ v(t) \\ y(t)}
	= \mat{cc}{A & B_d \\ C_v & D_{vd} \\ C_y & D_{yd} } \mat{c}{x(t) \\ d(t)},\quad
	\mat{c}{x(t_k+1) \\ v(t_k) \\ y(t_k)}
	= \mat{cc}{A_J & B_{Jd} \\ C_{Jv} & D_{Jvd}  \\ C_{Jy} & D_{Jyd}} \mat{c}{x(t_k) \\ d(t_k)}
\end{equation}
for all $t \in \N_0 \setminus \{t_1, t_2, \dots \}$ and all $k \in \N$ and reads as follows.

\begin{theorem}
	\label{DI::theo::es_main_svar}
	There exists an estimator \eqref{DI::eq::es_con} for the impulsive system \eqref{DI::eq::es_sys_alternative} such that their closed-loop interconnection with output $e = v - u$ achieves an energy gain smaller than $\ga$ for all sequences $(t_k)_{k\in \N_0}$ with \eqref{DI::eq::RDT} if there exist matrices $\Xb_0, \dots, \Xb_{T_{\max}} \in \S^{2n}$, $G_0, \dots G_{T_{\max}-1}$, $H_0, \dots, H_{T_{\max}-1}$, $S_0, \dots, S_{T_{\max}-1}$, $G_{JT_{\min}}, \dots, G_{JT_{\max}}$, $H_{JT_{\min}}, \dots , H_{JT_{\max}}$, $S_{JT_{\min}}, \dots, S_{JT_{\max}}\in \R^{n\times n}$ and $\smat{K & L \\ M & N} \in \R^{(n+n_u) \times (n+n_y)}$ satisfying
	\begin{subequations}
		\label{DI::theo::eq::es_lmis_svar}
		\begin{equation}
			S_k - G_k = S_{Jl} - G_{Jl} \text{ for all }j \in [T_{\min}, T_{\max}]\cap \N_0, \quad \Xb_k \cg 0,
		\end{equation}
		\begin{equation}
			(\bullet)^\top \mat{ccc|c}{0 & I & 0 & 0 \\ I & \Xb_{k+1} - \Gb_k - \Gb_k^\top & 0 & 0 \\ 0 & 0 & -\Xb_{k} & 0 \\ \hline 0 & 0 & 0 & P_\ga}
			\mat{ccc}{0 & \Ab_k & \Bb_k \\ I & 0 & 0 \\ 0 & I & 0 \\ \hline 0 & \Cb_k & \Db_k \\ 0 & 0 & I} \cl 0,
		\end{equation}
		\begin{equation}
			(\bullet)^\top \mat{ccc|c}{0 & I & 0 & 0 \\ I & \Xb_0 - \Gb_{Jk} - \Gb_{Jk}^\top & 0 & 0 \\ 0 & 0 & -\Xb_{k} & 0 \\ \hline 0 & 0 & 0 & P_\ga}
			\mat{ccc}{0 & \Ab_{Jk} & \Bb_{Jk} \\ I & 0 & 0 \\ 0 & I & 0 \\ \hline 0 & \Cb_{Jk} & \Db_{Jk} \\ 0 & 0 & I} \cl 0
		\end{equation}
	\end{subequations}
	for all indices $k$ contained in $[0, T_{\max}-1]\cap \N_0$, $[0, T_{\max}]\cap \N_0$, $[0, T_{\max}-1]\cap \N_0$ and $[T_{\min}, T_{\max}]\cap \N_0$, respectively.
	Here, the matrices $\smat{\Ab_k & \Bb_k \\ \Cb_k & \Db_k}$, $\Gb_k$, $\smat{\Ab_{Jk} & \Bb_{Jk} \\ \Cb_{Jk} & \Db_{Jk}}$ and $\Gb_{Jk}$ are defined as
	\begin{equation*}
		\mat{cc|c}{H_k A & H_k A & H_k B_d \\ G_k A & G_k A & G_k B_d \\ \hline C_v & C_v & D_v}
		+ \mat{cc}{0 & 0 \\ I & 0 \\ \hline 0 & -I}
		\mat{cc}{K & L \\ M & N}
		\mat{cc|c}{I & 0 & 0 \\ C_y & C_y & D_{yd}}, \quad
		\mat{cc}{H_k & H_k \\ S_k & G_k},
	\end{equation*}
	and
	\begin{equation*}
		\mat{cc|c}{H_{Jk} A_J & H_{Jk} A_J & H_{Jk} B_{Jd} \\ G_{Jk} A_J & G_{Jk} A_J & G_{Jk} B_{Jd} \\ \hline C_{Jv} & C_{Jv} & D_{Jvd}}
		+ \mat{cc}{0 & 0 \\ I & 0 \\ \hline 0 & -I}
		\mat{cc}{K & L \\ M & N}
		\mat{cc|c}{I & 0 & 0 \\ C_{Jy} & C_{Jy} & D_{Jyd}}
		\teq{ and }
		\mat{cc}{H_{Jk} & H_{Jk} \\ S_{Jk} & G_{Jk}},
	\end{equation*}
	respectively. Finally, if the inequalities \eqref{DI::theo::eq::es_lmis_svar} are feasible, an estimator \eqref{DI::eq::es_con} with the desired properties is obtained, e.g., by choosing
	\begin{equation*}
		\mat{cc}{A^e & B^e \\ C^e & D^e} := \mat{cc}{S_1 - G_1 & 0 \\ 0 & I}^{-1}\mat{cc}{K  & L \\ M & N}.
	\end{equation*}
\end{theorem}

\vspace{1ex}

Note that the inequalities \eqref{DI::theo::eq::es_lmis_svar} aren't LMIs, but this trouble can easily be overcome
by applying the Schur complement.
Moreover, feasibility of the inequalities \eqref{DI::theo::eq::es_lmis_svar} implies that the clock and slack variable based analysis criteria in Theorem \ref{DI::theo::clock_sv} are satisfied for the corresponding closed-loop interconnection. The slack variables in \eqref{DI::theo::eq::clockSV_Lyapunov_LMIsb} depend on the index $k$ and are structured as
\begin{equation*}
	\mat{cc}{G_k & S_1 - G_1 \\ H_k - G_k & -(S_1 - G_1)},
\end{equation*}
i.e., they have a constant/common second block column; the slack variables in \eqref{DI::theo::eq::clockSV_Lyapunov_LMIsc} are structured analogously with an identical second block column.
Finally, if the inequalities \eqref{DI::theo::eq::es_lmis_svar} are feasible, the obtained estimator \eqref{DI::eq::es_con} will have at most McMillian degree $n$, which can be much smaller than the one of the estimator resulting from Theorem \ref{DI::theo::es_main}.
However, the most important benefit of our new Theorem \ref{DI::theo::es_main} is its flexibility in allowing for extensions to systems involving multiple types of uncertainties in a much simpler fashion. Similarly as illustrated in the previous section, these extensions result even in computationally less demanding synthesis LMIs.

\begin{remark}[On the Conservatism]
	Both Theorems \ref{DI::theo::es_main} and \ref{DI::theo::es_main_svar} involve conservatism that is inherited from the employed closed-loop analysis criteria.
	Due to Lemma \ref{DI::lem::conservativsm_estimate_lifting} and the third statement of Theorem  \ref{DI::theo::clock}, both are linked to the clock based performance test in Theorem \ref{DI::theo::clock}. As we have mentioned and shown, this clock based test is in general conservative and the path based approach in Theorem \ref{DI::theo::path_Lyapunov} can lead to better results.
	If we take the clock based test as a starting point, then any additional conservatism of the criteria in Theorem~\ref{DI::theo::es_main} is exactly due to the relation in Lemma \ref{DI::lem::conservativsm_estimate_lifting}. Recall from the previous example section that this conservatism is expected to asymptotically vanish if the length of the involved filter increases to infinity.
	On the other hand, the additional conservatism of Theorem \ref{DI::theo::es_main_svar} is due to the introduction of slack variables with enforced structure, i.e., they are required to have a constant and common second block column, and there is no clear mechanism to reduce the conservatism introduced by this limitation.
\end{remark}

\subsection{Example}\label{DI::sec::exa_es}

As an illustration, let us consider an open-loop impulsive system \eqref{DI::eq::es_sys_alternative} with describing matrices given by
\begin{equation}
	\label{DI::eq::exa_syn}
	\mat{c|c}{A & B_d \\ \hline C_e & D_{ed} \\ \hdashline C_y & D_{yd}}
	= \mat{cc|c}{0.18 & 0.34 & -0.02 \\[0.75ex] -0.58 & 1.08 & -0.01 \\[0.75ex] \hline 0 & 1 & 0 \\ \hdashline 1 & 0 & 0}
	\teq{ and }
	\mat{c|c}{A_J & B_{Jd} \\ \hline C_{Je} & D_{Jed} \\ \hdashline C_{Jy} & D_{Jyd}}
	= \mat{cc|c}{0.47 & 0.41 & 0 \\[0.75ex] -0.01 & -0.02 & 1.32 \\[0.75ex] \hline 0 & 1 & 0 \\ \hdashline 0 & 0 & 0},
\end{equation}
which can also expressed as a feedback interconnection \eqref{DI::eq::es_sys} involving the impulsive operator \eqref{DI::eq::sys::del}.
Hence, we can design LTI estimators \eqref{DI::eq::es_con} achieving optimal upper bounds on the robust $\ell_2$-gain from the disturbance $d$ to the estimation error based on solving the inequalities in Theorem \ref{DI::theo::es_main} or \ref{DI::theo::es_main_svar} and by minimizing the upper bound~$\ga$.

Several of these upper bounds are given in Table \ref{DI::tab::es} for some dwell-time boundaries $T_{\min}$, $T_{\max}$ and some lengths $\nu$ of the filter \eqref{DI::eq::filter_TF_SS} involved in Theorem \ref{DI::theo::es_main}.
For this example, we observe that Theorem \ref{DI::theo::es_main} yields better upper bounds on the robust $\ell_2$-gain that even improve with increasing filter length $\nu$ if compared to the ones obtained by Theorem \ref{DI::theo::es_main_svar}. However, we also found some examples where the latter yields better upper bounds, which is possible since both results rely on different robust analysis criteria with different sources of conservatism.

Figure~\ref{DI::fig::es_exa} illustrates the to-be-estimated output $v$ of the system \eqref{DI::eq::es_sys_alternative} with \eqref{DI::eq::exa_syn} and $[T_{\min}, T_{\max}] = [9, 10]$ in response to the generalized disturbances
\begin{equation*}
	d_1(t) := 4\chi_{[0, 60]}(t) - 2 \chi_{(60, 200]}(t)
	\teq{and}
	d_2(t) := -2\cos(\tfrac{t}{2})\chi_{[0, 60]}(t) + 3 \sin(\tfrac{t}{4})\chi_{(60, 200]}(t)
\end{equation*}
at the left and right column, respectively; here, $\chi_{[a, b]}$ denotes the characteristic function which equals one on $[a, b]$ and vanishes elsewhere. It also shows the estimate $u$ of an LTI estimator constructed from Theorem~\ref{DI::theo::es_main} with a filter \eqref{DI::eq::filter_TF_SS} of length $\nu = 1$, from Theorem~\ref{DI::theo::es_main_svar} and from \texttt{hinfsyn} by ignoring the system's impulsive component in the top, middle and bottom row, respectively. 
Let us stress that the latter (nominal) design requires stability of the matrix $A$ in the flow component of \eqref{DI::eq::exa_syn} which is in stark contrast to Theorems~\ref{DI::theo::es_main} and \ref{DI::theo::es_main_svar}.
For this example, we observe that the estimators obtained from Theorems~\ref{DI::theo::es_main} and \ref{DI::theo::es_main_svar} admit roughly the same behavior and, despite being standard LTI systems, they approximate well the output $v$ of the impulsive system \eqref{DI::eq::es_sys_alternative} with \eqref{DI::eq::exa_syn}. In contrast, the estimator obtained from \texttt{hinfsyn} is not able to deal properly with the impulsive nature of the given system and provides only a poor approximation.

\begin{table}
	\caption{Upper bounds on the robust $\ell_2$-gain obtained by Theorems \ref{DI::theo::es_main} and \ref{DI::theo::es_main_svar} for the system \eqref{DI::eq::es_sys_alternative} with \eqref{DI::eq::exa_syn} and for some dwell-time boundaries $T_{\min}$, $T_{\max}$ and filter lengths $\nu$. }
	\label{DI::tab::es}
	\begin{center}
		\setlength{\tabcolsep}{7pt}%
		\renewcommand{\arraystretch}{1.2}%
		\begin{tabular}{ccccc}
			&Theorem \ref{DI::theo::es_main_svar} & \multicolumn{3}{c}{Theorem \ref{DI::theo::es_main}} \\ \cmidrule(l{1ex}){3-5}
			$(T_{\min}, T_{\max})$ & & $\nu = 1$ & $\nu = 2$ & $\nu = 3$ \\ \hline
			$(4, 5)$ & $3.574$ & $3.063$ & $2.513$ & $2.439$ \\ 
			$(5, 7)$ & $3.055$ & $2.655$ & $2.266$ & $2.147$ \\
			$(7, 9)$ & $2.239$ & $2.062$ & $1.950$ & $1.872$ \\
			$(9, 10)$ & $1.816$ & $1.755$ & $1.730$ & $1.709$ \\ \hline
		\end{tabular}
	\end{center}
\end{table}

\begin{figure}
	\begin{minipage}{0.49\textwidth}
		\includegraphics[width=\textwidth, trim=35 110 30 10, clip]{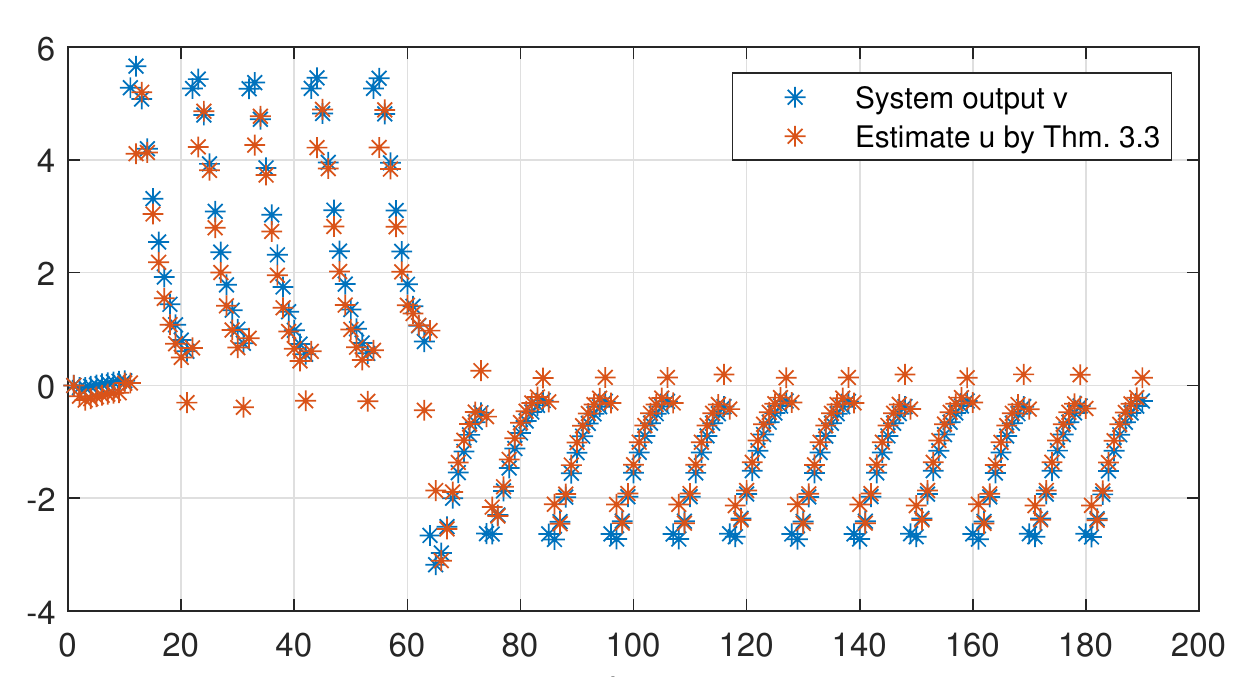}
	\end{minipage}
	\hfill
	\begin{minipage}{0.49\textwidth}
		\includegraphics[width=\textwidth, trim=35 110 30 10, clip]{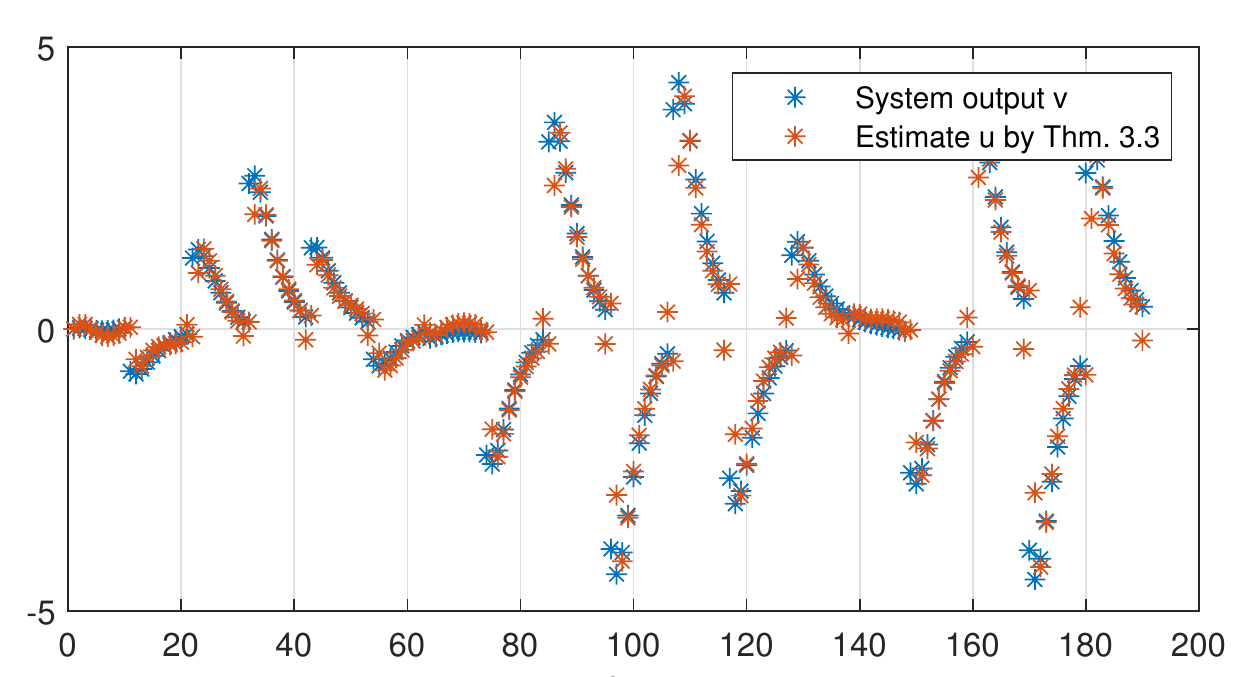}
	\end{minipage}
	\begin{minipage}{0.49\textwidth}
		\includegraphics[width=\textwidth, trim=35 100 30 10, clip]{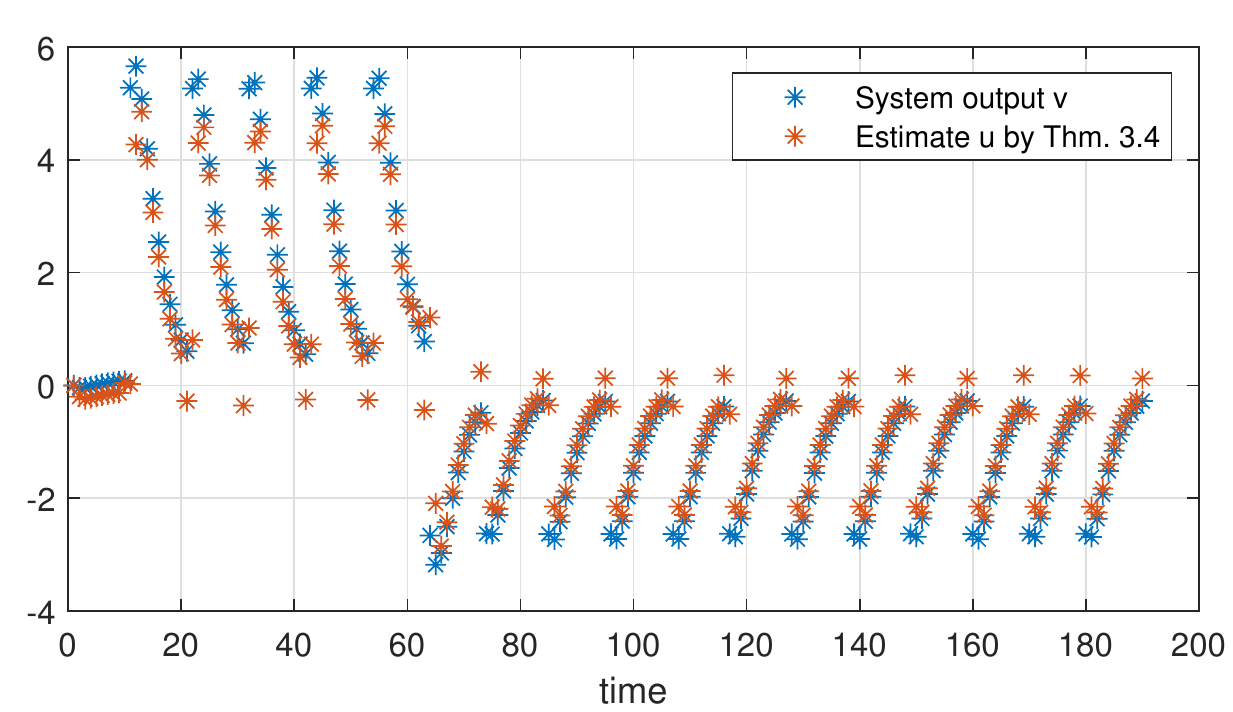}
	\end{minipage}
	\hfill
	\begin{minipage}{0.49\textwidth}
		\includegraphics[width=\textwidth, trim=35 110 30 10, clip]{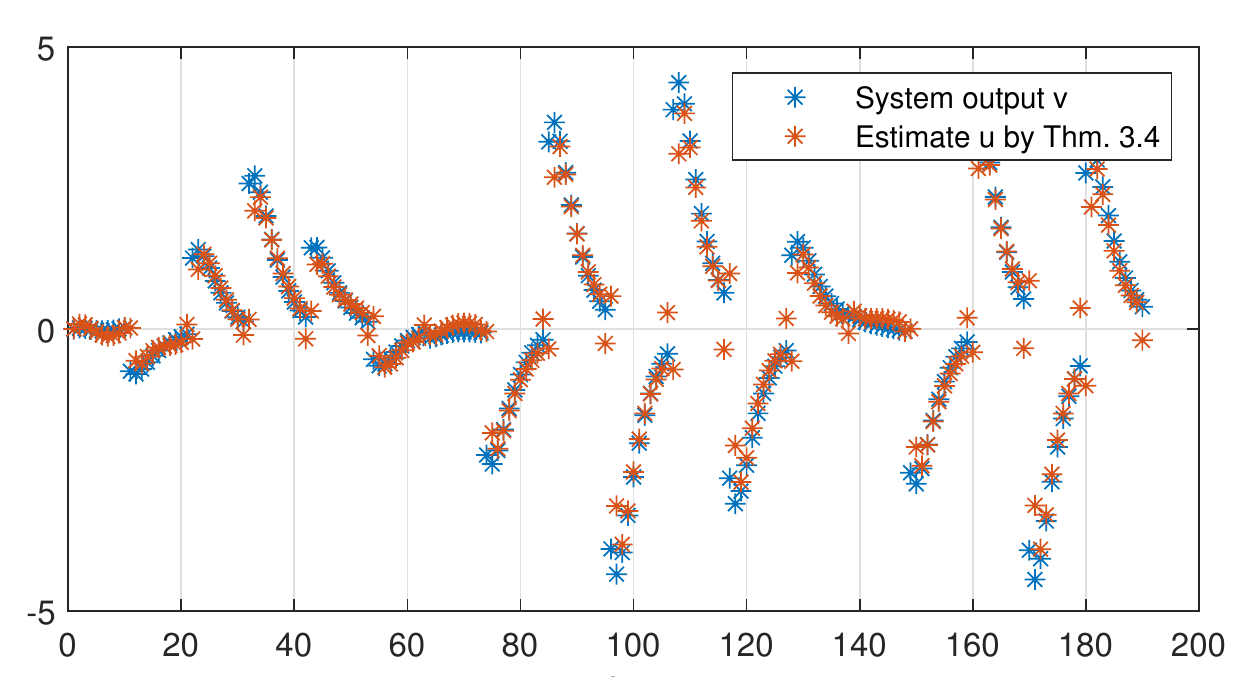}
	\end{minipage}
	\begin{minipage}{0.49\textwidth}
		\includegraphics[width=\textwidth, trim=35 110 30 10, clip]{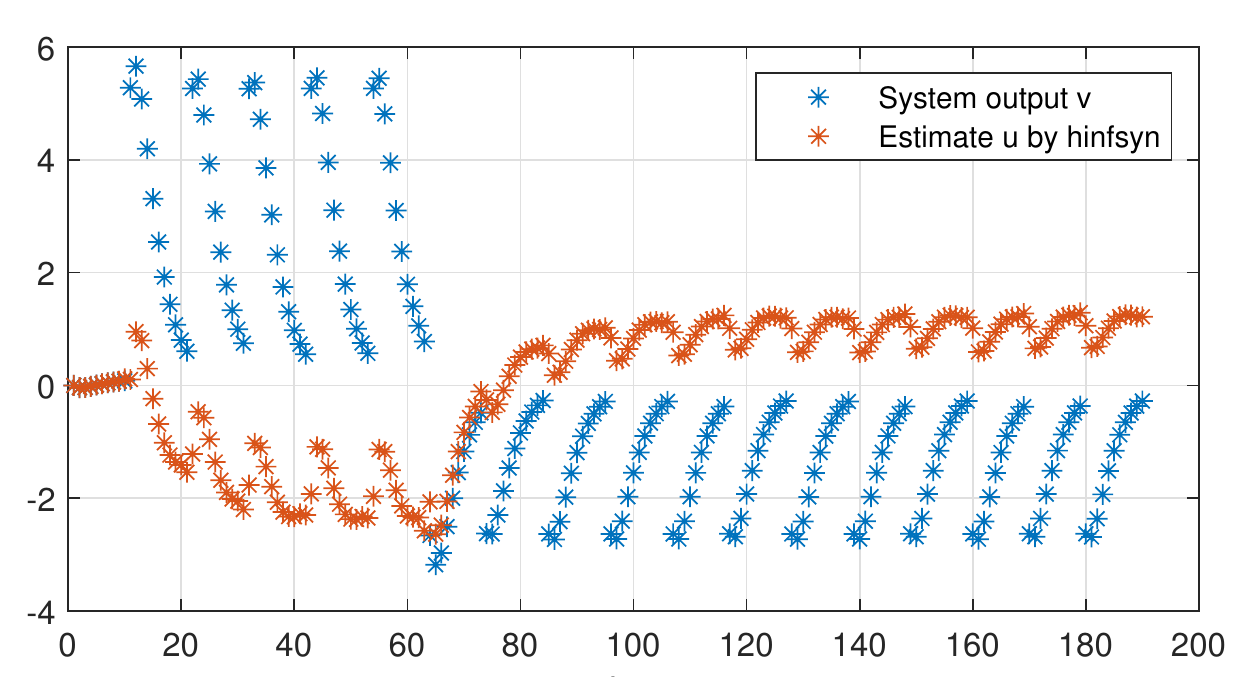}
	\end{minipage}
	\hfill
	\begin{minipage}{0.49\textwidth}
		\includegraphics[width=\textwidth, trim=35 100 30 10, clip]{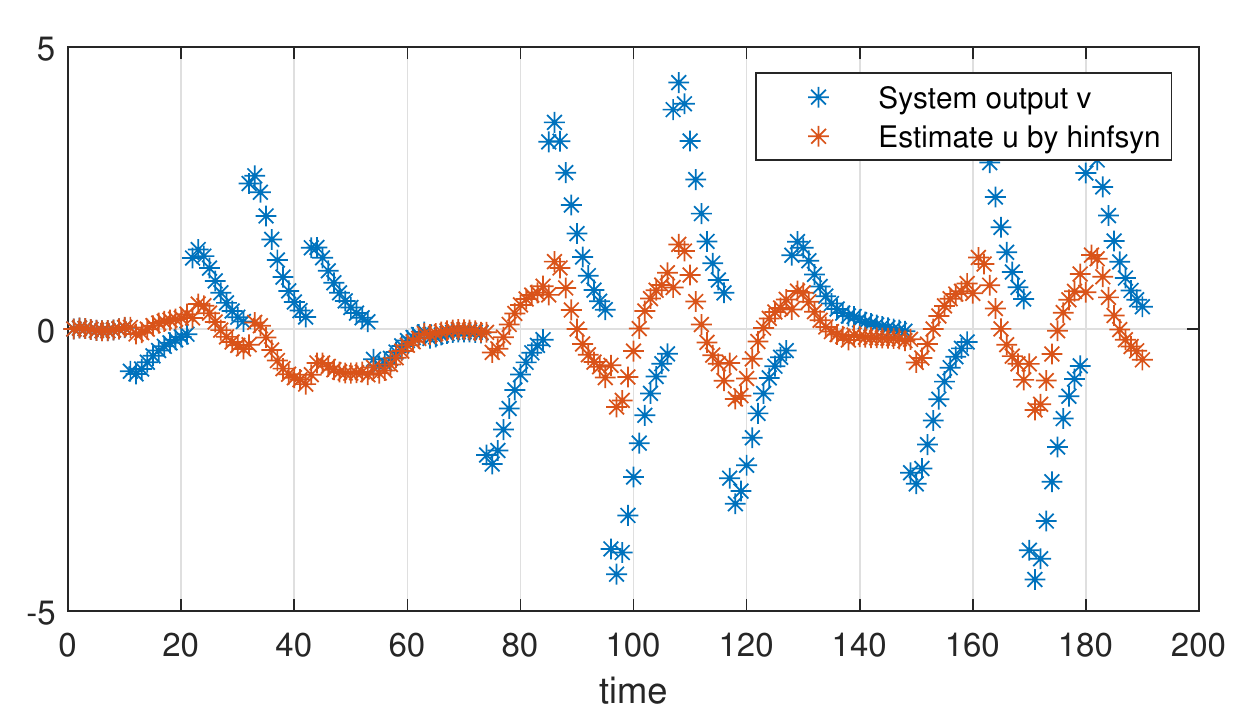}
	\end{minipage}
	\caption{To-be-estimated output $v$ of the system \eqref{DI::eq::es_sys_alternative} with \eqref{DI::eq::exa_syn} and $[T_{\min}, T_{\max}] = [9, 10]$ in response to the disturbances $d_1$ (left column) and $d_2$ (right column) as well as the estimate $u$ generated by an LTI estimator obtained from Theorem~\ref{DI::theo::es_main} (top row), Theorem~\ref{DI::theo::es_main_svar} (middle row) and from \texttt{hinfsyn} by ignoring the system's jump component (bottom row).}
	\label{DI::fig::es_exa}
\end{figure}

\section{Conclusions}

We consider linear discrete-time impulsive systems and employ IQCs for their analysis. This is achieved by viewing those systems as feedback interconnections of some non-impulsive linear discrete-time system with an impulsive operator and by constructing dedicated IQCs for the latter operator.
Similarly as the related approach in \cite{Hol22}, this not only enables extensions to impulsive systems affected by several types of uncertainties, but also permits us to tackle challenging controller design problems in which the full state might be unavailable for control.
Here, we consider the challenging problem of estimating the response of an impulsive system by means of a non-impulsive estimator.

Future research could for example deal with analogous results for impulsive operators whose impulse instants depend on parts of the system's state.
Moreover, obtaining similar results for impulsive systems with a continuous-time flow component would be useful and does not seem immediate in this particular context.



\bibliography{literatur}


\begin{appendix}
	
	\section{Proofs of Recapitulated and Modified Results}
	
	\begin{proof}[Proof of Theorem \ref{DI::theo::basic_Lyapunov}]
		Let $x(0) \in \R^n$, $d\in \ell_2^{n_d}$ and $(t_k)_{k\in \N_0}$ with \eqref{DI::eq::RDT} be arbitrary and let $x, e$ be the corresponding trajectory of the system \eqref{DI::eq::sys_standard}. By strictness of the finitely many inequalities \eqref{DI::theo::eq::basic_Lyapunov_LMIs}, we can find constants $\alpha, \beta,\eps >0$ and $\rho\in (0, 1)$ with $\alpha I \cle X \cle \beta I$ and such that \eqref{DI::theo::eq::basic_Lyapunov_LMIsb} persists to hold if we replace $-X$ and $R$ by $-\rho X$ and $R + \eps I$, respectively.

		Next, let us define the function $\eta: k \mapsto x(t_k+1)^\top Xx(t_k+1)$. Then we have, by the description \eqref{DI::eq::sys_standard}, by applying the discrete-time variation of constants formula several times, and by the modification of \eqref{DI::theo::eq::basic_Lyapunov_LMIsb},
		\begin{equation}
			\eta(k+1) - \rho \eta(k) + \sum_{t = t_k+1}^{t_{k+1}}(\bullet)^\top P\mat{c}{e(t) \\ d(t)} + \eps \|d(t)\|^2
			\leq 0
			%
			%
			%
			%
			\teq{ for all }k\in \N_0.
			\label{DI::pro::dissip1}
		\end{equation}
		\textit{Stability.} If $d = 0$, we infer from $Q \cge 0$ and the latter inequality that $\eta(k+1) - \rho \eta(k) \leq 0$ holds for all $k\in \N_0$. This implies
		\begin{equation*}
			\eta(k) \leq \rho^k \eta(0) \leq \beta \rho^k \|x(0)\|^2 \teq{ for all }k\in \N_0.
		\end{equation*}	
		Now let $t \in [t_k+1, t_{k+1}] \cap \N_0$ for some $k \in \N_0$ be arbitrary and let us abbreviate $\kappa := \max_{0 \leq j \leq T_{\max}} \|A^j\|^2$. Then we conclude
		\begin{equation*}
			\|x(t)\|^2 \leq \kappa \|x(t_k+1)\|^2 \leq \frac{\kappa}{\alpha} \eta(k) \leq \frac{\kappa\beta}{\alpha}  \rho^k \|x(0)\|^2
			\leq (M \hat \rho^t \|x(0)\|)^2
		\end{equation*}
		for $M:= \sqrt{\frac{\kappa \beta}{\alpha}}$ and $\hat \rho := \rho^{\frac{1}{2T_{\max}}} \in (0, 1)$. This shows stability of the system \eqref{DI::eq::sys_standard}.
		
		\noindent \textit{Quadratic Performance.} If $x(0) = 0$, we infer from \eqref{DI::pro::dissip1} and $X \cg 0$ that
		\begin{multline*}
			\sum_{t = 0}^{t_{k+1}}(\bullet)^\top P\mat{c}{e(t) \\ d(t)} + \eps \|d(t)\|^2
			= \sum_{l = 0}^k \left(	\sum_{t = t_l+1}^{t_{l+1}}(\bullet)^\top P\mat{c}{e(t) \\ d(t)} + \eps \|d(t)\|^2 \right) \\
			\leq \sum_{l = 0}^{k} \rho \eta(l) - \eta(l+1)
			\leq \sum_{l = 0}^{k} \eta(l) - \eta(l+1)
			= \eta(0) - \eta(k+1)
			\leq 0
		\end{multline*}
		holds for all $k\in \N_0$. This yields the claim by taking the limit $k\to \infty$.
	\end{proof}

	\vspace{2ex}
	
	\begin{proof}[Proof of Theorem \ref{DI::theo::path_Lyapunov}]
		For notational convenience we only show stability; quadratic performance is shown similarly is in the proof of Theorem \ref{DI::theo::basic_Lyapunov}.
		Let $x(0) \in \R^n$ and $(t_k)_{k\in \N_0}$ with \eqref{DI::eq::RDT} be arbitrary and let $x,$ be the corresponding state trajectory of the system \eqref{DI::eq::sys_standard} in response to $d = 0$. We can then find a sequence $(p^{(k)})_{k\in\N_0}$ with $p^{(k)} \in \Ps_L$ and $p^{(k+1)} \in \Ps_L^+(p^{(k)})$ for all $k\in \N_0$ such that
		\begin{equation}
			x(L(k+1)) = A_{p^{(k)}}x(Lk)
			\label{DI::eq::sysLA}
		\end{equation}
		holds for all $k\in \N_0$. Moreover, by strictness of the finitely many inequalities \eqref{DI::theo::eq::path:based:LMIs}, we can find constants $\alpha, \beta >0$ and $\rho\in (0, 1)$ with $\alpha I \cle X_p \cle \beta I$ for all $p \in \Ps_{L}$ and such that \eqref{DI::theo::eq::path:based:LMIsb} still holds if we replace $-X_p$ by $-\rho X_p$. Note that the left upper block of the latter inequality implies, by $Q \cge 0$,
		\begin{equation}
			\mat{c}{A_p \\ I}^\top \mat{cc}{X_q & 0 \\ 0 & -\rho X_p}\mat{c}{A_p \\ I} \cl 0
			\teq{ for all }q \in \Ps_{L}^+(p) \text{ ~and all~ } p \in \Ps_L.
			\label{DI::eq::sysLB}
		\end{equation}
		
		Next, let us define the function $\eta: k \mapsto x(Lk)^\top X_{p^{(k)}}x(Lk)$.
		Then we have, by the description \eqref{DI::eq::sysLA}, $p^{(k+1)} \in \Ps_L^+(p^{(k)})$ and by \eqref{DI::eq::sysLB},
		\begin{equation*}
			\eta(k+1) = (\bullet)^\top X_{p^{(k+1)}}x(L(k+1)) = (\bullet)^\top X_{p^{(k+1)}} A_{p^{(k)}} x(Lk)
			\leq \rho(\bullet)^\top X_{p^{(k)}} x(Lk) = \rho \eta(k)
			\teq{ for all }k\in \N_0.
		\end{equation*}
		Hence, we infer
		\begin{equation*}
			\eta(k) \leq \rho^k \eta(0) \leq \beta \rho^k \|x(0)\|^2 \teq{ for all }k\in \N_0.
		\end{equation*}	
		Now let $t \in [Lk, L(k+1)) \cap \N_0$ for some $k \in \N_0$ be arbitrary and let us abbreviate $\kappa := \max_{0 \leq i,j \leq L} \|A^j\|^2 \|A_J^i\|^2$. Then we conclude
		\begin{equation*}
			\|x(t)\|^2 \leq \kappa \|x(Lk)\|^2 \leq \frac{\kappa}{\alpha} \eta(k) \leq \frac{\kappa\beta}{\alpha} \rho^k \|x(0)\|^2
		\end{equation*}
		which proves the claim.
	\end{proof}

	\vspace{2ex}

	\begin{proof}[Proof of Theorem \ref{DI::theo::clock}]
		{\bfseries First Statement.}
		Let $x(0) \in \R^n$, $d\in \ell_2^{n_d}$ and $(t_k)_{k\in \N_0}$ with \eqref{DI::eq::RDT} be arbitrary and let $x, e$ be the corresponding trajectory of the system \eqref{DI::eq::sys_standard}. By strictness of the finitely many inequalities \eqref{DI::theo::eq::clock_Lyapunov_LMIs}, we can find constants $\alpha, \beta,\eps >0$ and $\rho\in (0, 1)$ with $\alpha I \cle X_k \cle \beta I$ for all $k \in [0, T_{\max}]\cap \N_0$ and such that \eqref{DI::theo::eq::clock_Lyapunov_LMIsb}, \eqref{DI::theo::eq::clock_Lyapunov_LMIsc} remain to hold if we replace $-X$ and $R$ by $-\rho X$ and $R + \eps I$, respectively

		Next, let us define the function $\eta: t \mapsto x(t)^\top X_{\theta(t)}x(t)$ involving the clock \eqref{DI::eq::clock}. Then we have, by the description \eqref{DI::eq::sys_standard}, the definition of $\theta$ and the modification of \eqref{DI::theo::eq::clock_Lyapunov_LMIsb},
		\begin{equation*}
			\eta(t+1) - \rho \eta(t)
			= (\bullet)^\top \mat{cc}{X_{\theta(t+1)} & 0 \\ 0 & -\rho X_{\theta(t)}}\mat{cc}{A & B \\ I & 0}\mat{c}{x(t) \\ d(t)}
			\leq - (\bullet)^\top P\mat{c}{e(t) \\ d(t)} - \eps \|d(t)\|^2
			%
			%
		\end{equation*}
		for all $t \in \N_0 \setminus \{t_1, t_2, \dots \}$. Similarly, we have, by the modification of \eqref{DI::theo::eq::clock_Lyapunov_LMIsc},
		\begin{equation*}
			\eta(t_k+1) - \rho \eta(t_k)
			= (\bullet)^\top \mat{cc}{X_0 & 0 \\ 0 & -\rho X_{\theta(t_k)}}\mat{cc}{A_J & B_J \\ I & 0}\mat{c}{x(t_k) \\ d(t_k)}
			\leq - (\bullet)^\top P\mat{c}{e(t_k) \\ d(t_k)} - \eps \|d(t_k)\|^2
		\end{equation*}
		for all $k \in \N$. In summary, we get
		\begin{equation}
			\label{DI::eq::pro::clock}
			\eta(t+1) - \rho \eta(t) + (\bullet)^\top P\mat{c}{e(t) \\ d(t)} + \eps \|d(t)\|^2
			\leq 0
			\teq{ for all }t \in \N_0.
		\end{equation}
		\textit{Stability.} If $d = 0$, we infer from $Q \cge 0$ and the latter inequality that $\eta(t+1) - \rho \eta(t) \leq 0$ holds for all $t \in \N_0$. This implies
		\begin{equation*}
			\eta(t) \leq \rho^t \eta(0)
			\teq{ for all } t \in \N_0.
		\end{equation*}
		Then we conclude
		\begin{equation*}
			\|x(t)\|^2 \leq \frac{1}{\alpha} \eta(t) \leq \frac{1}{\alpha}\rho^t \eta(0) \leq \frac{\beta}{\alpha} \rho^t \|x(0)\|^2
			\teq{ for all }t \in \N_0
		\end{equation*}
		which yields stability of the system \eqref{DI::eq::sys_standard}.
		
		\noindent\textit{Quadratic Performance.} If $x(0) = 0$, we infer from \eqref{DI::eq::pro::clock} and $X_k \cg 0$ for all $k$ that
		\begin{equation*}
			\sum_{t = 0}^{k}(\bullet)^\top P\mat{c}{e(t) \\ d(t)} + \eps \|d(t)\|^2
			\leq \sum_{t = 0}^{k} \rho \eta(t) - \eta(t+1)
			\leq \sum_{t = 0}^{k} \eta(t) - \eta(t+1)
			= \eta(0) - \eta(k+1)
			\leq 0
		\end{equation*}
		holds for all $k\in \N_0$. This yields the claim by taking the limit $k\to \infty$.
		
		\noindent{\bfseries Second Statement.} We show that $X := X_0$ satisfies the desired inequalities \eqref{DI::theo::eq::basic_Lyapunov_LMIs}, where \eqref{DI::theo::eq::basic_Lyapunov_LMIsa} is immediate from \eqref{DI::theo::eq::clock_Lyapunov_LMIsa}. To this end, let $k \in [T_{\min}, T_{\max}]$ be arbitrary and let us define the matrices
		\begin{equation*}
			U_0 := \mat{ccccc}{I & 0 & 0 & \cdots & 0\\ 0 & I& 0 & \cdots & 0} \teq{ as well as }U_l := \mat{cccccccc}{A^l & A^{l-1}B & \cdots & B & 0 & 0 & \cdots & 0 \\ 0 & 0 & \cdots & 0 & I & 0 & \cdots & 0}
		\end{equation*}
		for $l \in \{1, \dots, k\}$ which all have the same dimensions. Note that these matrices satisfy
		\begin{equation*}
			\mat{cc}{I & 0}U_{l+1} = \mat{cc}{A & B} U_l
			\teq{ for all }l\in \{0, \dots, k-1\} \teq{ and }  \mat{cc}{I & 0}U_{0} = \mat{cccc}{I & 0 & \cdots & 0}.
		\end{equation*}
		These matrices permit us to express the left-hand side of \eqref{DI::theo::eq::basic_Lyapunov_LMIsb} for $X=X_0$ as
		\begin{equation*}
			\begin{aligned}
				&\phantom{=}~(\bullet)^\top X_0 \mat{cc}{A_J & B_J}U_k - (\bullet)^\top X_0 \mat{cc}{I & 0} U_0 + \sum_{l = 0}^{k-1} (\bullet)^\top P \mat{cc}{C & D \\ 0 & I}U_l + (\bullet)^\top P \mat{cc}{C_J & D_J \\ 0 & I} U_k \\
				&= (\bullet)^\top\! \mat{cc|c}{X_0 & 0 & \\ 0 & -X_k & \\ \hline && P}\! \mat{cc}{A_J & B_J \\ I & 0 \\ \hline C_J & D_J \\ 0 & I}\! U_k + \underbrace{(\bullet)^\top X_k \mat{cc}{I & 0}U_k- (\bullet)^\top X_0 \mat{cc}{I & 0} U_0 + \sum_{l = 0}^{k-1} (\bullet)^\top\! P \mat{cc}{C & D \\ 0 & I}\!U_l}_{=: \Omega}.
			\end{aligned}
		\end{equation*}
		Via a telescoping sum and by the properties of $U_l$, the matrix $\Omega$ equals
		\begin{equation*}
			\arraycolsep=2pt
			\sum_{l = 0}^{k-1}  \left[(\bullet)^\top  X_{l+1}\mat{cc}{I & 0}U_{l+1} - (\bullet)^\top X_l \mat{cc}{I & 0}U_l +   (\bullet)^\top P \mat{cc}{C & D \\ 0 & I}U_l\right]
			= \sum_{l=0}^{k-1}(\bullet)^\top \mat{cc|c}{X_{l+1} & 0 & \\ 0 & -X_l & \\ \hline && P}\mat{cc}{A & B \\ I & 0 \\ \hline C & D \\ 0 & I} U_l.
		\end{equation*}
		This shows that the left-hand side of \eqref{DI::theo::eq::basic_Lyapunov_LMIsb} for $X=X_0$ equals
		\begin{equation*}
			\arraycolsep=1pt
			(\bullet)^{\!\top} \diag\left( \hspace{-0.5ex}(\bullet)^{\!\top} \!\mat{cc|c}{X_{1} & 0 & \\ 0 & -\!X_0 & \\ \hline && P}\hspace{-1ex}\mat{cc}{A & B \\ I & 0 \\ \hline C & D \\ 0 & I}\hspace{-0.8ex},
			\dots,
			(\bullet)^{\!\top} \!\mat{cc|c}{X_{k} & 0 & \\ 0 & -\!X_{k-1} & \\ \hline && P}\hspace{-1ex}\mat{cc}{A & B \\ I & 0 \\ \hline C & D \\ 0 & I}\hspace{-0.8ex},   (\bullet)^{\!\top} \!\mat{cc|c}{X_0 & 0 & \\ 0 & -\!X_k & \\ \hline && P}\hspace{-1ex} \mat{cc}{A_J & B_J \\ I & 0 \\ \hline C_J & D_J \\ 0 & I} \hspace{-0.9ex}\right) \mat{c}{U_0 \\ \vdots \\U_k}\!.
		\end{equation*}
		Since the inner part of this matrix is negative definite by \eqref{DI::theo::eq::clock_Lyapunov_LMIsb} and \eqref{DI::theo::eq::clock_Lyapunov_LMIsc}
		and since $\smat{U_0 \\ \vdots \\ U_k}$ has full column rank, we can conclude that \eqref{DI::theo::eq::basic_Lyapunov_LMIsb} holds, which finishes the proof.
		
		\noindent {\bfseries Third Statement.} We only sketch this part by showing it in the case of $T_{\min} = T_{\max} = 1$. The general case is obtained similarly, but involves taking more Schur complements. Note at first that the middle block of \eqref{DI::theo::eq::basic_Lyapunov_LMIsb} for $k=1$ implies $(\bullet)^\top P \smat{D \\ I} \cl 0$ by $X \cg 0$ and $Q = (\bullet)^\top P \smat{I \\ 0} \cge 0$.
		Since $P$ is invertible and again due to $Q = (\bullet)^\top P \smat{I \\ 0} \cge 0$, we can apply the dualization lemma (see, e.g., \cite{SchWei00})
		to reformulate \eqref{DI::theo::eq::basic_Lyapunov_LMIsb}  as
		\begin{equation}
			(\bullet)^\top \t P \mat{c}{0 \\ I} \cle 0
			\teq{ and }
			(\bullet)^\top \mat{cc|cc}{Y_0 & 0 & \\ 0 & -Y_0 & \\ \hline &&\t P & 0 \\ && 0 & \t P}\mat{ccc}{I & 0 & 0 \\ -A^\top A_J^\top & - C^\top & -A^\top C_J^\top \\[0.1ex] \hline
				-\t B^\top A_J^\top & -\t D^\top & -\t B^\top C_J^\top \\
				- \t B_J^\top & 0 & -\t D_J^\top} \cg 0
			\label{DI::pro::eq::dual_lifted}
		\end{equation}
		for $Y_0 := X^{-1}$, $\t P := P^{-1}$, $\t B := (0, B)$, $\t D := (-I, D)$, $\t B_J := (0, B_J)$ and $\t D_J := (-I, D_J)$.
		
		Next, observe that we can infer $\t D \t P \t D^\top - CY_0C^\top \cg 0$ from the $(2, 2)$ block of
		\eqref{DI::pro::eq::dual_lifted}. Together with $Y_0 \cg 0$ and $\t B \t P \t B^\top = (\bullet)^\top \t P \smat{0 \\ I} B^\top \cle 0$, we can define
		\begin{equation*}
			Y_1 :=AY_0A^\top - \t B \t P \t B^\top + \big(\t B\t P \t D^\top - A Y_0 C^\top\big) \big(\t D\t P \t D^\top - C Y_0 C^\top \big)^{-1} (\bullet)^\top + \eps I
		\end{equation*}
		and conclude that this matrix is positive definite for any $\eps > 0$.
		This particular choice assures that
		\begin{equation}
			(\bullet)^\top \mat{cc|c}{Y_1 & 0 & \\ 0 & -Y_0 & \\ \hline && \t P} \mat{cc}{I & 0 \\ -A^\top & -C^\top \\ \hline 0 & I \\ -B^\top & -D^\top} \cg 0
			\label{DI::eq::pro::dual_flow}
		\end{equation}
		holds. Indeed, by the Schur complement and $\t D \t P \t D^\top - CY_0C^\top \cg 0$, the latter inequality is equivalent to
		\begin{equation*}
			0 \cl Y_1 -AY_0A^\top + \t B \t P \t B^\top - \big(\t B \t P \t D^\top - A Y_0 C^\top \big) \big(\t D \t P \t D^\top - CY_0C^\top \big)^{-1}(\bullet)^\top
			=  \eps I
		\end{equation*}
		which is trivially true.
		
		Moreover, this choice also guarantees, for some small enough $\eps > 0$,
		\begin{equation}
			(\bullet)^\top \mat{cc|c}{Y_0 & 0 & \\ 0 & -Y_1 & \\ \hline && \t P}\mat{cc}{I & 0 \\ -A_J^\top & -C_J^\top \\ \hline 0 & I \\ -B_J^\top & - D_J^\top} \cg 0.
			\label{DI::eq::pro::dual_jump}
		\end{equation}
		Indeed, by Schur, the definition of $Y_1$ and a few computations, the latter inequality is equivalent to
		\begin{equation*}
			(\bullet)^\top \mat{cc|cc}{Y_0 & 0 & \\ 0 & -Y_0 & \\ \hline && \t P & 0 \\ && 0 & \t P}\mat{ccc}{I & 0 & 0\\ -A^\top A_J^\top & -A^\top C_J^\top & -C^\top \\ \hline -\t B^\top A_J^\top & - \t B^\top C_J^\top & -\t D^\top \\ -\t B_J^\top  & -\t D_J^\top & 0} - \eps (\bullet)^\top \mat{ccc}{A_J^\top & C_J^\top & 0} \cg 0
		\end{equation*}
		and the claim follows with \eqref{DI::pro::eq::dual_lifted}.
		
		Finally, we can apply the dualization lemma for the inequalities \eqref{DI::eq::pro::dual_flow} and \eqref{DI::eq::pro::dual_jump} to infer
		\begin{equation*}
			X_0\cg 0,\quad X_1 \cg 0,\quad
			(\bullet)^\top \mat{cc|c}{X_1 & 0 & \\ 0 & -X_0 & \\ \hline && P}\mat{cc}{A & B \\ I & 0 \\ \hline C & D \\ 0 & I} \cl 0
			\text{ ~and~ }
			(\bullet)^\top \mat{cc|c}{X_0 & 0 & \\ 0 & -X_1 & \\ \hline && P}\mat{cc}{A_J & B_J \\ I & 0 \\ \hline C_J & D_J \\ 0 & I} \cl 0
		\end{equation*}
		for $X_0 := Y_0^{-1}$ and $X_1 := Y_1^{-1}$. As desired, this leads to
		\eqref{DI::theo::eq::clock_Lyapunov_LMIsb} and \eqref{DI::theo::eq::clock_Lyapunov_LMIsc} for $k\in[0,T_{\max}-1]=[0,0]$ and $k\in [T_{\min},T_{\max}]=[1,1]$, respectively.
	\end{proof}
	
	\vspace{2ex}
	
	\begin{proof}[Proof of Theorem \ref{DI::theo::IQC}]
		Let $x(0) \in \R^n$ and $d\in \ell_2^{n_d}$ be arbitrary and let $x, \xi, z, w, u, y, e$ be the corresponding trajectory of the interconnection \eqref{DI::eq::sys_intercon} and the filter \eqref{DI::eq::filter}. By strictness of the inequality \eqref{DI::theo::eq::maina}, we can find some small $\eps > 0$ such that this inequality is still valid if we replace $-X$ by $-X + \eps I$ and $R$ by $R + \eps I$.
		
		Next, let us define the functions $\eta: t \mapsto (\bullet)^\top X \smat{\xi(t)\\ x(t)}$ and $\nu : t_k+1 \mapsto \xi(t_k+1)^\top Z(k)\xi(t_k+1)$. Then we have, by the dynamics of the augmented system \eqref{DI::eq::sys_augmented} and by the modification of \eqref{DI::theo::eq::maina},
		\begin{multline*}
			\eta(t_{k}+1) -\eta(0) + \eps \sum_{t = 0}^{t_{k}}\left\|\mat{c}{x(t)\\d(t)}\right\|^2
			= \sum_{t = 0}^{t_{k}}\eta(t+1) -\eta(t) + \eps \left\|\mat{c}{x(t)\\d(t)}\right\|^2 \\
			\stackrel{\eqref{DI::theo::eq::maina}}{\leq} -\sum_{t = 0}^{t_{k}} y(t)^\top M y(t) -\sum_{t = 0}^{t_{k}} (\bullet)^\top P \mat{c}{e(t) \\ d(t)} 
			\stackrel{\eqref{DI::eq::pointwise_IQC}}{\leq}  \nu(t_{k}+1)-\sum_{t = 0}^{t_{k}} (\bullet)^\top P \mat{c}{e(t) \\ d(t)}  
		\end{multline*}
		for all $k \in \N$. This gives
		\begin{equation}
			\label{DI::eq::pro::IQC}
			\sum_{t = 0}^{t_{k}} (\bullet)^\top P \mat{c}{e(t) \\ d(t)} + \eps \sum_{t = 0}^{t_{k}} \left\|\mat{c}{x(t)\\d(t)}\right\|^2  \leq \eta(0) - (\eta - \nu)(t_k+1) \stackrel{\eqref{DI::theo::eq::mainb}}{\leq} \eta(0)
			\teq{ for all } k \in  \N.
		\end{equation}
		\textit{Stability.} If $d = 0$, we infer from $Q \cge 0$ and the latter inequality that
		\begin{equation*}
			\eps \sum_{t = 0}^{t_k} \|x(t)\|^2 \leq \eta(0)
			\teq{holds for all}k\in \N.
		\end{equation*}
		This yields the claim by letting $k$ approach infinity.
		
		\noindent\textit{Quadratic Performance.} If $x(0) = 0$, we have $\eta(0) = 0$ since $\xi(0) = 0$. Then we infer from \eqref{DI::eq::pro::IQC} that
		\begin{equation*}
			\sum_{t = 0}^{t_{k}} (\bullet)^\top P \mat{c}{e(t) \\ d(t)} \leq - \eps \sum_{t = 0}^{t_k} \|d(t)\|^2 - \eps \sum_{t = 0}^{t_k} \|x(t)\|^2
			\leq - \eps \sum_{t = 0}^{t_k} \|d(t)\|^2
			\teq{ holds for all }k\in \N.
		\end{equation*}
		This yields the claim by taking the limit $k \to \infty$.
	\end{proof}

	%
	%
	%
	%
	%
	%
	%

\end{appendix}

\end{document}